\documentclass[12pt]{article}

\usepackage[usenames]{color} 
\usepackage{times,pifont,graphicx,color,amsmath,amsthm,amsfonts,amssymb,url}
\usepackage{comment}
\usepackage{hyperref}
\usepackage{bbm}

\theoremstyle{theorem}
\newtheorem{theorem}{Theorem}
\newtheorem{lemma}{Lemma}
\newtheorem{proposition}{Proposition}
\newtheorem{conjecture}{Conjecture}
\theoremstyle{definition}
\newtheorem*{definition}{Definition}
\newtheorem*{remark}{Remark}

\def\ed{\stackrel{d}{=}}

\def\E{ {\mathbb E }}
\def\P{ {\mathbb P }}

\DeclareMathOperator{\Shi}{Shi}

\includecomment{problem}

\includecomment{solution}

\title{Time and place of the maximum for one-dimensional diffusion bridges and meanders}
\author{Robin Khanfir}
\date{	Ecole Normale Sup{\'e}rieure, Paris\\
				robin.khanfir@ens.fr\\[2ex]
	July 22, 2018
}

\begin{document}
\maketitle

\begin{abstract}

For three constrained Brownian motions, the excursion, the meander, and the reflected bridge,  the densities of the maximum and of the time to reach it were expressed as double series by Majumdar, Randon-Furling, Kearney, and Yor (2008). Some of these series were regularized by Abel summation. Similar results for Bessel processes were obtained by Schehr and Le Doussal (2010) using the real space renormalization group method. Here this work is reviewed, and extended from the point of view of
one-dimensional diffusion theory to some other diffusion processes including skew Brownian bridges and generalized Bessel meanders. We discuss the limits of the application of this method for other diffusion processes.

\end{abstract}

\section{Introduction}

Problems of studying the maximum and the time when it is reached for a random process, denoted $(M,\rho)$, arise from numerous models in economics, biology, or data science for example. The law of the maximum of some standard diffusion processes is used in statistics to understand the quality of an estimator (such as the law of the maximum of a reflected Brownian bridge which is the basis of the Kolmogorov-Smirnov test). Moreover, understanding those laws builds some connections between probabilities and complex analysis: for example the maximum of a Brownian excursion is closely related to the Riemann zeta function. See \cite{526597220011001} for a review of various probabilistic interpretations of this function. Finally, interest in the laws of the time for various processes to achieve their maximum is driven by their surprising and often counter-intuitive answers. The first of this kind of results is the case of the Brownian motion and the Brownian bridge on $[0,1]$. These are the 'arcsine law' and the 'uniform law' found by Levy in \cite{edsnum.CM.1940..7..283.019400101}.
\begin{eqnarray*}
\text{Brownian motion}:\ \frac{\P(\rho\in du)}{du} &=& \frac{1}{\pi\sqrt{u(1-u)}} \\
\text{Brownian bridge}:\ \frac{\P(\rho\in du)}{du} &=&1 
\end{eqnarray*}

See also \cite{edsarx.1804.0789620180101} for more background and some generalizations.

In \cite{majumdar}, Satya. N. Majumdar, Julien Randon-Furling, Michael J. Kearney, and Marc Yor have computed the joint density of the place and time of the maximum, $(M,\rho)$, for three constrained Brownian motions: the standard excursion, the Brownian meander, and the standard reflected Brownian bridge. They have used two methods. The first one relies on a physical argument, based on the idea that the statistical weight of a path is proportional to a propagator defined with a Hamiltonian and a potential (describing the constraints). Shortly said, it is a path integral method. The second method uses some 'agreement formulas' and decomposition at the maximum (as discussed in Section 2). These formulas are identities in law with the following form, where $U$ and $ V$ are independent positive random variables, and $c,\ \mu$ some reals, and $F$ is an arbitrary non-negative measurable function. 

\begin{equation}
\label{agreement_majumdar}
\E[F(M^2,\rho)]=c\E\left[F\left(\frac{1}{U+V},\frac{U}{U+V}\right)(U+V)^\mu\right]
\end{equation}

The law of $U$ and $V$ and the values of $c$ and $\mu$ depend on the case according to the following table:

\[\begin{array}{c|c|c|c}
 & \text{Excursion} & \text{Meander} & \text{Reflected bridge} \\
\hline
c & \displaystyle{\sqrt{\frac{\pi}{2}}} & \displaystyle{\sqrt{\frac{\pi}{8}}} & \displaystyle{\sqrt{\frac{\pi}{2}}}\\
\hline
\mu & \displaystyle{\frac{1}{2}} & \displaystyle{-\frac{1}{2}} & \displaystyle{-\frac{1}{2}}\\
\hline
\E(e^{-\lambda U}) & \displaystyle{\frac{\sqrt{2\lambda}}{\sinh(\sqrt{2\lambda})}} & \displaystyle{\frac{\sqrt{2\lambda}}{\sinh(\sqrt{2\lambda})}} & \displaystyle{\frac{1}{\cosh(\sqrt{2\lambda})}}\\
\hline
\E(e^{-\lambda V}) & \displaystyle{\frac{\sqrt{2\lambda}}{\sinh(\sqrt{2\lambda})}} & \displaystyle{\frac{2}{\sqrt{2\lambda}}\tanh\left(\frac{\sqrt{2\lambda}}{2}\right)} & \displaystyle{\frac{1}{\cosh(\sqrt{2\lambda})}}
\end{array}\]

The authors of \cite{majumdar} verified that the two methods give the same expression for the density of $(M,\rho)$ in each case, and confirmed their results numerically with high precision. The proofs of the above formulas can be found explicitly in the literature for the case of the excursion and reflected Brownian bridge, in \cite{decompo} for example. However, in \cite{majumdar}, it is pointed out the same can not be said for the case of the meander. Here we present a theorem of decomposition at the maximum, then apply it explicitly to derive these formulas, and so we focus ourselves on the method based on the agreement formula. 

In 2010, Schehr and Le Doussal also derived the joint density of $(M,\rho)$ in \cite{000274266600015n.d.} thanks to an other method called the real space renormalization group. This method is powerful enough to study not only the Brownian excursion, meander, and reflected bridge, but also Bessel bridges. It also shares with the path integral method the quality to be more physically intuitive than the agreement formula method. However, both of these methods raise some issues of technical rigor. So we prefer to work here with the rigorous methods of one-dimensional diffusion theory.

Moreover, the authors of \cite{majumdar} and \cite{000274266600015n.d.} deduced expressions for the densities of $\rho$ and $M$ for each case, as a double series. To do so, they integrated the joint density of $(M,\rho)$. Especially in \cite{majumdar}, it was obtained thanks to a version of the agreement formula (\ref{agreement_majumdar}) in the terms of densities and after expanding the densities of $U$ and $V$ as a series. This yielded the following expressions, with $z>0$ and $0<u<1$.

\begin{align*}
\textbf{Excursion}\\
  \frac{\P(M\in dz)}{dz} & =  \frac{2^{\frac{3}{2}}\pi^{\frac{5}{2}}}{z^4}\sum_{m,n=1}^\infty(-1)^{m+n}\frac{m^2 n^2}{m^2 -n^2}\left(e^{-\frac{n^2\pi^2}{2z^2}}-e^{-\frac{m^2\pi^2}{2z^2}}\right) \\
\frac{\P(\rho\in du)}{du} & =  3\sum_{m,n=1}^\infty(-1)^{m+n}\frac{m^2 n^2}{[n^2 u +m^2 (1-u)]^{5/2}}\\
\textbf{Meander}\\
  \frac{\P(M\in dz)}{dz} & =  \frac{2^{\frac{3}{2}}\pi^{\frac{1}{2}}}{z^2}\sum_{m,n=1}^\infty((-1)^{m+n}-(-1)^n)\frac{n^2}{m^2 -n^2}\left(e^{-\frac{n^2\pi^2}{2z^2}}-e^{-\frac{m^2\pi^2}{2z^2}}\right) \\
\frac{\P(\rho\in du)}{du} & =  2\sum_{m=0,n=1}^\infty(-1)^{n+1}\frac{n^2}{[n^2 u +(2m+1)^2 (1-u)]^{3/2}}\\
\textbf{Reflected bridge}\\
  \frac{\P(M\in dz)}{dz} & =  \frac{2^{\frac{3}{2}}\pi^{\frac{1}{2}}}{z^2}\sum_{m,n=0}^\infty(-1)^{m+n}\frac{(m+\frac{1}{2}) (n+\frac{1}{2})}{(m+\frac{1}{2})^2 -(n+\frac{1}{2})^2}\left(e^{-\frac{(n+1/2)^2\pi^2}{2z^2}}-e^{-\frac{(m+1/2)^2\pi^2}{2z^2}}\right) \\
\frac{\P(\rho\in du)}{du} & =  2\sum_{m,n=0}^\infty(-1)^{m+n}\frac{(2m+1)(2n+1)}{[(2n+1)^2 u +(2m+1)^2 (1-u)]^{3/2}}\\
\textbf{Bessel bridge of dimension }\delta>0\\
\frac{\P(M\in dz)}{dz} & = \frac{4}{C_\nu z^{1+\delta}}\sum_{\substack{m,n\geq 1}}\frac{j_{\nu,m}^{\nu+1}}{J_{\nu+1}(j_{\nu,m})}\frac{j_{\nu,n}^{\nu+1}}{J_{\nu+1}(j_{\nu,n})}\frac{e^{-\frac{j_{\nu,n}^2}{2z^2}}-e^{-\frac{j_{\nu,m}^2}{2z^2}}}{j_{\nu,m}^2-j_{\nu,n}^2} \\
\frac{\P(\rho\in du)}{du} & =2\delta\sum_{m,n\geq1}\frac{j_{\nu,m}^{\nu+1}}{J_{\nu+1}(j_{\nu,m})}\frac{j_{\nu,n}^{\nu+1}}{J_{\nu+1}(j_{\nu,n})}\frac{1}{[j_{\nu,n}^2 u+j_{\nu,m}^2 (1-u)]^{\nu+2}}
\end{align*}

The notations used for the case of Bessel bridges are precised in Section 4.1.

Some of these series are not absolutely convergent. In \cite{majumdar}, it is explained that their meaning should by given by a regularization by adding some $\alpha^n$ terms, or said with other words, with the non-standard Abel summation. Abel summation can give values to some divergent series but in general a divergent series are not Abel summable either. Even if it those series was indeed Abel summable, their Abel summations do not need to be equal to the desired densities. As these points were not addressed in \cite{majumdar} nor \cite{000274266600015n.d.}, we shall make precise a sense of these summations before proving them.

First, in Section 2, we will present a family of known results, called agreement formulas, linked with a decomposition at the maximum for diffusion processes. We will give some examples in Section 3, including (\ref{agreement_majumdar}) as a particular case. In Section 4, from the agreement formula, we will rigorously derive the above expressions for the densities of $M$ and $\rho$ for all Bessel bridges. Those expressions have been properly proven in \cite{MR1701890} (Section 11) by Pitman and Yor for $M$ and Bessel bridges of index $\nu<-1/2$, but our results are a generalization, and our rigorous proof of the expression for the density for $\rho$ seems to be new.

To deduce these results we will employ three tools already provided by others: the agreement formula, the Abel summation, and a series expansion of the density of the first hitting time of $1$. We will discuss our use of these tools and under which conditions the method might or might not be extended to a wider class of processes. We also provide similar formulas for skew Brownian bridges in Section 5, for which the same method can not be applied. In Section 6, we provide some formulas for the case of generalized Bessel meanders.

\section{Decomposition at the maximum}

Let $X$ be a one dimensional regular diffusion process on $I\supset[0,\infty)$ with infinite lifetime, such that the infinitesimal generator is $A=D_m D_s$, with $s$ scale function and $m$ speed function. There is a jointly continuous density relative to $m$, $p(t,x,y)=\P_x(X_t\in dy)/m(dy)$.

Let also $T_z=\inf\{t\geq0,\ X_t=z\}$, $M_t=\sup_{0\leq u\leq t}X_u$, and $\rho_t=\inf\{u\geq0,\ X_u=M_t\}$. It is known $T_z$ has a continuous density relative to $dt$, $f_{xz}(t)=\P_x(T_z\in dt)/dt$. In the following, we will denote $T$ the random variable $T_1$ under $\P_0$ and its density $f=f_{01}$. See \cite{sample} for the precise definitions and the proofs of those facts.

Williams first gave a path decomposition at the time where the maximum is reached for diffusion processes with finite maximum. A lot of mathematicians then generalized, applied, and completed these ideas. For our use, we take a variation on density factorization on a finite time interval. It was given by Fitzsimmons in an unpublished manuscript then presented in \cite{decompo} as the Theorem 2. Various proofs of the different points of the theorem could be found in \cite{williams} and \cite{csaki}.
\begin{theorem}
For $0<u<1$, $x,y\leq z <\infty$ and $x,y\in I$, we have
\begin{enumerate}
\item 
\begin{equation}
\label{density_agreement}
\P_x(M_t\in dz,\rho_t\in du, X_t\in dy)=f_{xz}(u)f_{yz}(t-u)s(dz)m(dy)du
\end{equation}
\item 
\begin{equation}
\label{density_agreement_bridge}
\P_x(M_t\in dz,\rho_t\in du | X_t=y)=\frac{f_{xz}(u)f_{yz}(t-u)}{p(t,x,y)}s(dz)du
\end{equation}
\item Under $\P_x$ conditionally given $M_t=z$, $\rho_t=u$, and $X_t=y$ the path fragments $(X_v,0\leq v\leq u)$ and $(X_{t-v}, 0\leq v\leq t-u)$ are independent, respectively of law $(X_v,0\leq v\leq T_z)$ under $\P_x$ given $T_z=u$ and $(X_v,0\leq v\leq T_z)$ under $\P_y$ given $T_z=t-u$.
\end{enumerate}
\end{theorem}

As said in \cite{decompo}, if $\forall z\geq x$, $T_z<\infty$ $\P_x$-a.s., for any $x,y$ the theorem gives two equivalent definitions of a $\sigma$-finite measure on $\mathcal{C}_0(\mathbb{R}_+)$:
\begin{enumerate}
	\item pick $t$ according to $p(t,x,y)dt$ and run an $X$-bridge of length $t$ to $x$ from $y$.
	\item pick $z$ according to $s(dz)$ restricted to $(x\vee y,\infty)$, run to independent copies of $X$ till $T_z$, one starting from $x$ and the other from $y$, then put them back to back.
\end{enumerate}

\begin{remark}
When X is a $3$-dimensional Bessel process, this measure corresponds to Ito's excursion law. While 1. is conditioning on the length, 2. is conditioning on the maximum.
\end{remark}

\subsection{Standard bridge of a diffusion process with Brownian scaling}

In the following, $x=y=0$ and $t=1$, $X^{\rm br}$ is an $X$-bridge of length $1$ from $0$ to $0$ and we suppose $\forall z>0,\ T_z<\infty$ a.s., and $(X_t)\ed(\sqrt{c}X_{t/c})$ for any $c>0$. The scaling implies for $z>y\geq0$
\[s(dz)=s'(z)dz \text{ and }s'(z)=s'(1)z^\mu\text{ (see \cite{lamperti})}\] \[f_{yz}(u)=\frac{1}{z^2}f_{\frac{y}{z}1}\left(\frac{u}{z^2}\right)\text{ and }\E_y(e^{-\lambda T_z})=E_{\frac{y}{z}}(e^{-\lambda z^2 T_1})\]

Under those hypothesis, the agreement formula (\ref{density_agreement_bridge}) becomes
\begin{equation}
\label{density_agreement_standard_bridge}
\P(M^{\rm br}\in dz,\rho^{\rm br}\in du)=\frac{1}{p(1,0,0)}f\left(\frac{u}{z^2}\right)f\left(\frac{1-u}{z^2}\right)\frac{s'(z)}{z^4}dz du
\end{equation}

The following proof, inspired from the previous equivalence, constructs a process by rescaling the definition 2. and comparing it with the standard bridge. It is a direct generalization of the proof in \cite{decompo} for Bessel bridges. In other words, the only useful hypothesis to adapt the proof of Pitman and Yor are the ones above. The case for Bessel bridges from \cite{decompo} is also presented later.

\textbf{Construction}: Take $X$ and $\hat{X}$ two independent copies starting at $0$, respectively hitting $1$ for the first time at $T$ and $\hat{T}$. Putting them back to back, define

\[\tilde{X}_t=\begin{cases}
								X_t &\text{ if }t\leq T\\
								\hat{X}_{T+\hat{T}-t} &\text{ if }T\leq t\leq T+\hat{T}
							\end{cases}
\]

Finally, let $\displaystyle{\tilde{X}_u^{\rm br}=\frac{1}{\sqrt{T+\hat{T}}}X_{u(T+\hat{T})}}$ for $0\leq u\leq1$ and
\[\tilde{M}^{\rm br}=\frac{1}{\sqrt{T+\hat{T}}}\text{ and }\tilde{\rho}^{\rm br}=\frac{T}{T+\hat{T}}\]
Comparing (\ref{density_agreement_standard_bridge}) with
\[\P(\tilde{M}^{\rm br}\in dz,\tilde{\rho}^{\rm br}\in du)=2f\left(\frac{u}{z^2}\right)f\left(\frac{1-u}{z^2}\right)\frac{1}{z^5}dz du\]
We have the change of measure:
\begin{equation}
\label{agreement_standard_bridge}
E[F((M^{\rm br})^2,\rho^{\rm br})]=\frac{1}{2p(1,0,0)}\E\left[F\left(\frac{1}{T+\hat{T}},\frac{T}{T+\hat{T}}\right)\frac{1}{\sqrt{T+\hat{T}}}s'\left(\frac{1}{\sqrt{T+\hat{T}}}\right)\right]
\end{equation}
We use the last point of the theorem to extend the change of measure for the paths.
\begin{equation}
\label{change_measure}
\E[\Psi(X^{\rm br})]=\frac{1}{2p(1,0,0)}\E[\Psi(\tilde{X}^{\rm br})\tilde{M}^{\rm br}s'(\tilde{M}^{\rm br})]
\end{equation}

where $F$ and $\Psi$ are arbitrary non-negative measurable functions.
\begin{remark}
Starting from the above equality of expectancies, and assuming the scaling property, a simple change of variable get us the agreement formula as in Theorem 1 for the case of bridges from $0$ to $0$, and reversely. The assumption $T_z<\infty$ is only needed to well define the construction, but not to have the equivalence between (\ref{density_agreement_standard_bridge}) and (\ref{agreement_standard_bridge}).

\end{remark}

\section{Applications of the agreement formula for some examples}

In this section, $F$ and $\Psi$ are arbitrary non-negative measurable functions, $z>0$, and $0<u<1$.

\subsection{Case of Bessel bridges}

This application is already presented with precise definition and more details in \cite{decompo}.

Let $\delta>0$, $X\ed$ BES$(\delta)$ and $\delta=2(\nu+1)$, we can choose $s(dx)=x^{1-\delta}dx$, $m(dx)=2x^{\delta-1}dx$. We have $\displaystyle{p(t,0,y)=(2t)^{-\frac{\delta}{2}}\Gamma\left(\frac{\delta}{2}\right)^{-1}\exp\left(-\frac{y^2}{2t}\right)}$

We denote
\[C_\nu:=\frac{1}{2p(1,0,0)}=2^{\frac{\delta}{2}-1}\Gamma\left(\frac{\delta}{2}\right)\]

For the case of a (standard) Bessel bridge of dimension $\delta$, we can write (\ref{density_agreement_standard_bridge}), (\ref{agreement_standard_bridge}), and (\ref{change_measure}) respectively
\begin{equation}
\label{agreement_bessel}
\P(M^{\rm br}\in dz,\rho^{\rm br}\in du)=\frac{2C_\nu}{z^{3+\delta}}f_\nu\left(\frac{u}{z^2}\right)f_\nu\left(\frac{1-u}{z^2}\right)dz du
\end{equation}

\[\E[F((M^{\rm br})^2,\rho^{\rm br})]=C_\nu\E\left[F\left(\frac{1}{T+\hat{T}},\frac{T}{T+\hat{T}}\right)(T+\hat{T})^{\nu}\right]\]

\[\E[\Psi(X^{\rm br})]=C_\nu\E[\Psi(\tilde{X}^{\rm br})(\tilde{M}^{\rm br})^{2-\delta}]\]

Moreover, we have the well-known Laplace transform
\[\E(e^{-\lambda T})=\frac{(2\lambda)^{\frac{\nu}{2}}}{C_\nu I_\nu(\sqrt{2\lambda})}\]
where $I_\nu$ is the modified Bessel function of the first kind of index $\nu$
\[I_\nu(z)=\sum_{k\geq0}\frac{1}{k!\Gamma(\nu+k+1)}\left(\frac{z}{2}\right)^{2k+\nu}\]

For $\delta=1$ we have the case of the reflected Brownian bridge. For $\delta=3$, we have the case of the Brownian excursion. Respectively,
\[\E(e^{-\lambda T^{(1)}})=\frac{1}{\cosh(\sqrt{2\lambda})}\text{ and }\E(e^{-\lambda T^{(3)}})=\frac{\sqrt{2\lambda}}{\sinh(\sqrt{2\lambda})}\]

\[C_{-1/2}=C_{1/2}=\sqrt{\frac{\pi}{2}}\]

We have retrieved the formulas (\ref{agreement_majumdar}) of \cite{majumdar}.

\subsection{Case of skew Brownian bridges}

Let $0<\beta<1$, $X$ a skew Brownian motion$(\beta)$ such that $\P(X_1>0)=\beta$. We define
\[\sigma(x)=\begin{cases}
							\frac{1}{\beta} &\text{ if }x\geq0\\
							\frac{1}{1-\beta} &\text{ if }x<0
						\end{cases}
\]
and we can choose $s(x)=\sigma(x)x$ (then $s'(x)=\sigma(x)$) and $m(dx)=\frac{2}{\sigma(x)}dx$. See \cite{skew} for definitions and more details on skew Brownian motions.

For $a>0$, $\P(X_1>a)=2\beta\P(B_1>a)$ and $\P(X_1<-a)=2(1-\beta)\P(B_1<-a)$. So $p(1,0,y)=\frac{1}{\sqrt{2\pi}}\exp\left(-\frac{y^2}{2}\right)$ (it is the density of $B_1$ against $dy$). After application of the agreement formulas (\ref{density_agreement_standard_bridge}) and (\ref{agreement_standard_bridge}) for the case of a (standard) skew Brownian bridge

\begin{equation}
\label{agreement_skew}
\P(M^{\rm br}\in dz,\rho^{\rm br}\in du)=\frac{\sqrt{2\pi}}{\beta z^4}f_\beta\left(\frac{u}{z^2}\right)f_\beta\left(\frac{1-u}{z^2}\right)dz du
\end{equation}

\[\E[F((M^{\rm br})^2,\rho^{\rm br})]=\sqrt{\frac{\pi}{2}}\frac{1}{\beta}\E\left[F\left(\frac{1}{T+\hat{T}},\frac{T}{T+\hat{T}}\right)\frac{1}{\sqrt{T+\hat{T}}}\right]\]

To compute the Laplace transform of $T$, we consider the successive excursions above $1$ of a reflected Brownian motion and condition on which one is the first to remain positive for $X$.
\[\E(e^{-\lambda T})=\frac{\beta}{\cosh(\sqrt{2\lambda})-(1-\beta)e^{-\sqrt{2\lambda}}}\]
\begin{remark}
By applying the equality with $\beta = 1$, corresponding to the reflected Brownian bridge, we easily verify we retrieve the same formula as before. It is not surprising, as a consequence of a convergence in law.
\end{remark}

The case $\beta=1/2$ corresponds to the (non-reflected) Brownian bridge, and we have
\[\sqrt{\frac{\pi}{2}}\frac{1}{\beta}=\sqrt{2\pi}\text{ and }\E(e^{-\lambda T})=e^{-\sqrt{2\lambda}}\]

\subsection{Case of generalized Brownian meanders}

Let $k\in\mathbb{N}^*$, the $k$th Brownian meander is a continuous stochastic process $X^{k-\rm me}$ on $[0,1]$ such that $\displaystyle{\P(X_1^{k-\rm me}\in dy)=\frac{2}{2^{k/2}\Gamma(k/2)}y^{k-1}e^{-\frac{y^2}{2}}}$ for $y>0$ ($X_1^{k-\rm me}$ follows the chi distribution of dimension $k$) and conditionally given $X_1^{k-\rm me}=y$, $X^{k-\rm me}$ is a Bessel bridge of dimension $3$, from $0$ to $y$, and of length $1$. In particular, the cases $k=1,2,3$ are respectively the Brownian co-meander, the Brownian meander, and the $3$-dimensional Bessel process started at $0$ on $[0,1]$. See \cite{ucb.b1802599320070101} for details and background.

From (\ref{density_agreement_bridge}), the Brownian scaling, and after integrating with respect to $y$, we have
\[\P(M^{k-\rm me}\in dz,\rho^{k-\rm me}\in du)=2c_k\frac{z^k}{z^6}f\left(\frac{u}{z^2}\right)\phi_k\left(\frac{1-u}{z^2}\right)dz du\]

\[\text{ where }c_k=\frac{1}{k}\times\frac{2^{3/2}\Gamma(3/2)}{2^{k/2}\Gamma(k/2)}\text{ and }\phi_k(t)=\int_0^1 k x^{k-1}f_{x1}(t)dx\geq0\]

Here, $f_{x1}$ is the density of $T_1=T$ for the $3$-dimensional Bessel process started at $x$ and $f=f_{01}$. $\phi_k$ is the density of some random time $S_k\geq0$ because$\displaystyle{\int_0^\infty \phi_k(t)dt=1}$. We can chose $S_k$ independent from $T$ and write

\[\E[F((M^{k-\rm me})^2,\rho^{k-\rm me})]=c_k\E\left[F\left(\frac{1}{T+S_k},\frac{T}{T+S_k}\right)\left(\frac{1}{\sqrt{T+S_k}}\right)^{k-1}\right]\]

We already have $\displaystyle{\E(e^{-\lambda T^{(3)}})=\frac{\sqrt{2\lambda}}{\sinh(\sqrt{2\lambda})}}$ and we give the Laplace transform of $S_k$
\begin{align*}
\E(e^{-\lambda V_k})&=\int_0^\infty e^{-\lambda t}\phi_k(t)dt\\
&=\int_0^1 k x^{k-1}\E_x(e^{-\lambda T_1})dx\\
&=\frac{k}{\sinh(\sqrt{2\lambda})}\int_0^1 x^{k-2}\sinh(\sqrt{2\lambda}x)dx
\end{align*}

It is because $\displaystyle{\E_0(e^{-\lambda T_1})=\E_0(e^{-\lambda T_x})\E_x(e^{-\lambda T_1})=\E_0(e^{-\lambda x^2 T_1})\E_x(e^{-\lambda T_1})}$. Moreover for $\theta>0$, $n\in\mathbb{N}$, it is easy to prove by induction the following identities

\begin{align*}
\frac{\theta^{2n+1}}{(2n)!}\int_0^1 x^{2n}\sinh(\theta x)dx&=\cosh(\theta)\sum_{i=0}^n\frac{\theta^{2i}}{(2i)!}-\sinh(\theta)\sum_{i=0}^{n-1}\frac{\theta^{2i+1}}{(2i+1)!}-1\\
\frac{\theta^{2n+2}}{(2n+1)!}\int_0^1 x^{2n+1}\sinh(\theta x)dx&=\cosh(\theta)\sum_{i=0}^n\frac{\theta^{2i+1}}{(2i+1)!}-\sinh(\theta)\sum_{i=0}^{n}\frac{\theta^{2i}}{(2i)!}
\end{align*}

They can be written for any $m\in\mathbb{N}$ as

\[\frac{\theta^{m+1}}{m!}\int_0^1 x^{m}\sinh(\theta x)dx=\frac{e^{-\theta}}{2}\sum_{i=0}^m\frac{\theta^i}{i!}+(-1)^m\frac{e^\theta}{2}\sum_{i=0}^m\frac{(-\theta)^i}{i!}-\frac{1+(-1)^m}{2}\]

For the case $k=1$ we use the function $\Shi$, the hyperbolic sine integral
\[\Shi(x)=\int_0^x\frac{\sinh(t)}{t}dt\]

For the cases $k=1,2,3$, i.e the Brownian co-meander, the classical Brownian meander, and $3$-dimensional Bessel process

\begin{align*}
c_1&=1 & & \E(e^{-\lambda S_1})=\frac{\Shi(\sqrt{2\lambda})}{\sinh(\sqrt{2\lambda})}\\
c_2&=\sqrt{\frac{\pi}{8}} & & \E(e^{-\lambda S_2})=\frac{2}{\sqrt{2\lambda}}\tanh\left(\frac{\sqrt{2\lambda}}{2}\right)\text{ that is (\ref{agreement_majumdar}) for the Brownian meander}\\
c_3&=\frac{1}{3}& & \E(e^{-\lambda S_3})=\frac{3}{2\lambda\sinh(\sqrt{2\lambda})}(\sqrt{2\lambda}\cosh(\sqrt{2\lambda})-\sinh(\sqrt{2\lambda}))
\end{align*}

For $k=2$, we recognize $4S_2$ has the same law as the last zero of a reflected Brownian motion started at $0$ before hitting $1$.

For $k=3$, $\displaystyle{\E(e^{-\lambda T^{(5)}})\E(e^{-\lambda S_3})=\E(e^{-\lambda T^{(3)}})}$ where $T^{(5)}$ or $T^{(3)}$ is the first hitting time of $1$ respectively for the $5$-dimensional or the $3$-dimensional Bessel process started at $0$.

\section{Marginal densities for Bessel bridges}

In this section, we deduce the marginal densities of $M$ and $\rho$ from the agreement formula for Bessel bridges. The expressions are already presented in \cite{majumdar} for the reflected Brownian bridge and the Brownian excursion, and in \cite{526597220011001} for all Bessel bridges, after non-rigorous manipulations (see our introduction), but our demonstrations are new. To make rigorous the expressions of the densities of $M$ and $\rho$ from agreement formula we will use the same non-standard summation as in \cite{majumdar}. 

\begin{definition}Abel summation

Let $(u_n)$ a sequence of complex numbers. If $\displaystyle{\forall\alpha\in(0,1),\ \sum_{n\geq0}\alpha^n |u_n|<\infty}$ and if there is $S\in\mathbb{C}$ such that
\[\sum_{n\geq0}\alpha^n u_n\underset{\alpha\rightarrow1}{\longrightarrow}S\]
We say that $S$ is the Abel summation of $(u_n)$ and we denote
\[\sum_{n\geq0}u_n=S\ (\mathcal{A})\]
\end{definition}

For example, $\displaystyle{\sum_{n\geq0}(-1)^n=\frac{1}{2}\ (\mathcal{A})}$. This summation gives often a value for divergent series with alternating terms and has good properties. See \cite{hardy2000divergent} for details.

\begin{proposition}

Linearity: If $\displaystyle{\sum_{n\geq0}u_n =S\ (\mathcal{A})}$ and $\displaystyle{\sum_{n\geq0}v_n =T\ (\mathcal{A})}$ then $\displaystyle{\sum_{n\geq0}u_n+\lambda v_n =S+\lambda T\ (\mathcal{A})}$.

Stability: If $\displaystyle{\sum_{n\geq0}u_n =S\ (\mathcal{A})}$ then $\displaystyle{\sum_{n\geq0}u_{n+1} =S-u_0\ (\mathcal{A})}$

Regularity: If $\displaystyle{\sum_{n=0}^N u_n\underset{N\rightarrow\infty}{\longrightarrow}S}$ then $\displaystyle{\sum_{n\geq0}u_n =S\ (\mathcal{A})}$

\end{proposition}

\begin{definition}
Let $(u_{m,n})$ be a double  sequence of complex numbers. If $\displaystyle{\forall\alpha\in(0,1),\ \sum_{m,n\geq0}\alpha^{m+n}|u_{m,n}|<\infty}$ and if there is $S\in\mathbb{C}$ such that \[\sum_{m,n\geq0}\alpha_1^m \alpha_2^n u_{m,n}\underset{\substack{\alpha_1\rightarrow1^- \\ \alpha_2\rightarrow1^-}}{\longrightarrow}S\] we write
\[\sum_{m,n\geq0}u_{n,m}=S\ (\mathcal{A})\]
In particular, in this case $\displaystyle{\sum_{k\geq0}\left(\sum_{\substack{m,n\geq0 \\ m+n=k}}u_{m,n}\right)=S\ (\mathcal{A})}$
\end{definition}

\subsection{The formulas and their proof}
We use the notations: $\delta=2(\nu+1)>0$, $C_\nu=2^\nu\Gamma(\nu+1)$ and $j_{\nu,n}$ are the strictly positive zeros in increasing order of $J_\nu$, where $J_\nu$ is the Bessel function of the first kind of index $\nu$ defined by \[J_\nu(z)=\sum_{k\geq0}\frac{(-1)^k}{k!\Gamma(\nu+k+1)}\left(\frac{z}{2}\right)^{2k+\nu}\] We recall the agreement formula (\ref{agreement_bessel}) for a standard Bessel bridge of dimension $\delta$: for $z>0$ and $0<u<1$
\[\P(M\in dz,\rho\in du)=\frac{2C_\nu}{z^{3+\delta}}f_\nu\left(\frac{u}{z^2}\right)f_\nu\left(\frac{1-u}{z^2}\right)dz du\]
where $f_\nu$ is the density of the first hitting time of $1$ of a Bessel process of dimension $\delta$ started from $0$ (previously denoted $T$). From \cite{kent} and \cite{MR1701890}, we know a series expansion formula for $f_\nu$

\begin{equation}
\label{hitting_bessel}
f_\nu(t)=\frac{1}{C_\nu}\sum_{n\geq1}\frac{j_{\nu,n}^{\nu+1}}{J_{\nu+1}(j_{\nu,n})}e^{-\frac{j_{\nu,n}^2}{2}t}=\frac{1}{C_\nu}\sum_{n\geq1}(-1)^{n-1}g_{\nu,n}e^{-\frac{j_{\nu,n}^2}{2}t}
\end{equation}

This series expansion is a key-point for the following work. We will discuss its basis later. We precise some notations and asymptotic equivalents (see \cite{asymptotics} pages 237-242 and 247-248)
\[g_{\nu,n}=(-1)^{n-1}\frac{j_{\nu,n}^{\nu+1}}{J_{\nu+1}(j_{\nu,n})}>0\]
\[j_{\nu,n}\sim \pi n\]
\[g_{\nu,n}\sim \sqrt{\frac{\pi}{2}}(\pi n)^{\nu+3/2}\]

\begin{theorem}
The density of the max of a standard Bessel bridge is equal to
\[\frac{\P(M\in dz)}{dz}=\frac{2}{C_\nu z^{3+\delta}}\sum_{n\geq1}\left(\frac{j_{\nu,n}^{\nu+1}}{J_{\nu+1}(j_{\nu,n})}\right)^2e^{-\frac{j_{\nu,n}^2}{2z^2}}\]
\begin{equation}
\label{density_M_bessel}+\frac{4}{C_\nu z^{1+\delta}}\sum_{\substack{m,n\geq 1\\ m\neq n}}\frac{j_{\nu,m}^{\nu+1}}{J_{\nu+1}(j_{\nu,m})}\frac{j_{\nu,n}^{\nu+1}}{J_{\nu+1}(j_{\nu,n})}\frac{e^{-\frac{j_{\nu,n}^2}{2z^2}}-e^{-\frac{j_{\nu,m}^2}{2z^2}}}{j_{\nu,m}^2-j_{\nu,n}^2}\ (\mathcal{A})
\end{equation}
The convergence in $(\alpha_1,\alpha_2)\rightarrow(1^-,1^-)$ is uniform on $(0,\infty)$.
\end{theorem}

\begin{theorem}
The density of the argmax of a standard Bessel bridge is equal to
\begin{equation}
\label{density_rho_bessel}
\frac{\P(\rho\in du)}{du}=2\delta\sum_{m,n\geq1}\frac{j_{\nu,m}^{\nu+1}}{J_{\nu+1}(j_{\nu,m})}\frac{j_{\nu,n}^{\nu+1}}{J_{\nu+1}(j_{\nu,n})}\frac{1}{[j_{\nu,n}^2 u+j_{\nu,m}^2 (1-u)]^{\nu+2}}\ (\mathcal{A})
\end{equation}
The convergence in $(\alpha_1,\alpha_2)\rightarrow(1^-,1^-)$ is uniform on every segment of $(0,1)$.
\end{theorem}

\begin{remark}
We would like to integrate the agreement formula with respect for $z$ or $u$ to find the marginal densities. However, for example with $\delta=3$,
\[\pi^2\sum_{n\geq1}(-1)^{n-1}n^2\int_0^\infty\exp\left(-\frac{n^2\pi^2}{2}t\right)dt=2\sum_{n\geq0}(-1)^{n}\neq 1=\int_0^\infty f_{\frac{1}{2}}(t)dt\ ?!\]
The integration erases the small exponential term so the series cannot commute with the integral. It is why the Abel summation would be so useful because it adds new small terms independent from the integration and allows the commutation. We can remark also that with $\displaystyle{\sum_{n\geq0}(-1)^{n}=\frac{1}{2}\ (\mathcal{A})}$ the previous calculus works well.
\end{remark}

\begin{proof} For this proof, to simplify the notations, we could forget the indexes $\nu$.
Let us define for $0<\alpha<1$ and $t\geq0$: \[f^\alpha(t)=\frac{1}{C_\nu}\sum_{n\geq1}\alpha^{n-1}\frac{j_{\nu,n}^{\nu+1}}{J_{\nu+1}(j_{\nu,n})}e^{-\frac{j_{\nu,n}^2}{2}t}=\frac{1}{C_\nu}\sum_{n\geq1}(-1)^{n-1}\alpha^{n-1}g_n e^{-\frac{j_{n}^2}{2}t}\]

By normal convergence, $f^\alpha$ is continuous on $[0,\infty)$. We denote naturally $f^1=f$.

If $t>0$ and $0<\alpha\leq1$, $\displaystyle{|f^\alpha(t)|\leq \frac{1}{C_\nu}\sum_{n\geq1}g_n e^{-\frac{j_n^2}{2}t}}$. Thus, by dominated convergence, there is $t_0>0$ such that for any $0<\alpha\leq1$:
\begin{equation}
\label{bounded_infinity}
\text{If }t>t_0,\text{ then }|f^\alpha(t)|\leq e^{-\frac{j_0^2}{4}t}\underset{t\rightarrow +\infty}{\longrightarrow}0
\end{equation}

First we are going to show \[\frac{1}{z^{3+\delta}}\int_0^1 f^{\alpha_1}\left(\frac{u}{z^2}\right)f^{\alpha_2}\left(\frac{1-u}{z^2}\right)du\underset{\substack{\alpha_1\rightarrow1^- \\ \alpha_2\rightarrow1^-}}{\longrightarrow}\frac{1}{z^{3+\delta}}\int_0^1 f\left(\frac{u}{z^2}\right)f\left(\frac{1-u}{z^2}\right)du\] uniformly for $z\in(0,\infty)$ then we compute the LHS by inverting the series and the integral. We begin by a simple estimate for $t>0$

\begin{equation}
\label{inequality_uniform}
|f(t)-f^\alpha(t)|\leq\frac{1}{C_\nu}\sum_{n\geq1}|1-\alpha^{n-1}|g_n e^{-\frac{j_{n}^2}{2}t}\leq|1-\alpha|\times\frac{1}{C_\nu}\sum_{n\geq1}n g_n e^{-\frac{j_{n}^2}{2}t}
\end{equation}

So for any $r>0$, $\displaystyle{||f-f^\alpha||_{\infty,[r,\infty)}\leq|1-\alpha|\times\frac{1}{C_\nu}\sum_{n\geq1}n g_n e^{-\frac{j_{n}^2}{2}r}\underset{\alpha\rightarrow1}{\longrightarrow}0}$

Let us first admit a lemma we will prove later.

\begin{lemma}
\label{lemma}
For any $q>0$, the family $(f^\alpha)_{0<\alpha<1}$ is bounded in $L^2(\mathbb{R}_+, t^{q-1}dt)$ by a constant $C(q)>0$.
\end{lemma}

Remark the lemma (with (\ref{inequality_uniform})) implies $||f||_{L^2(\mathbb{R}_+,t^{q-1}dt)}\leq C(q)$. We will often apply the lemma with $q=1$, i.e. the family $(f^\alpha)_{0<\alpha\leq1}$ is bounded in $L^2(\mathbb{R}_+)$, so we will simply write $C(1)=C$.

Let $\varepsilon>0$. For any $0<\alpha_1,\alpha_2\leq1$ and $z>0$, by using (\ref{bounded_infinity}) and the lemma,

\begin{align*}
\left|\frac{1}{z^{3+\delta}}\int_0^{\frac{1}{2}} f^{\alpha_1}\left(\frac{u}{z^2}\right)f^{\alpha_2}\left(\frac{1-u}{z^2}\right)du\right|&\leq\frac{e^{-\frac{j_0^2}{8z^2}}}{z^{3+\delta}}\int_0^{\frac{1}{2}} \left|f^{\alpha_1}\left(\frac{u}{z^2}\right)\right|du\text{ for small }z\\
&\leq\frac{e^{-\frac{j_0^2}{8z^2}}}{\sqrt{2}z^{3+\delta}}z||f^{\alpha_1}||_{L^2(\mathbb{R}_+)}\\
&\leq C\frac{e^{-\frac{j_0^2}{8z^2}}}{\sqrt{2}z^{2+\delta}}\\
&\leq\varepsilon\text{ for small }z
\end{align*}

For big $z$, we can provide an other inequality,
\begin{align*}
\left|\frac{1}{z^{3+\delta}}\int_0^{1} f^{\alpha_1}\left(\frac{u}{z^2}\right)f^{\alpha_2}\left(\frac{1-u}{z^2}\right)du\right|&\leq\frac{1}{z^{3+\delta}}\sqrt{\int_0^1 f^{\alpha_1}\left(\frac{u}{z^2}\right)^2 du}\sqrt{\int_0^1 f^{\alpha_2}\left(\frac{1-u}{z^2}\right)^2 du}\\
&\leq\frac{1}{z^{3+\delta}}z^2||f^{\alpha_1}||_{L^2(\mathbb{R}_+)}||f^{\alpha_2}||_{L^2(\mathbb{R}_+)}\\
&\leq\frac{C^2}{z^{1+\delta}}\\
&\leq\varepsilon\text{ for big }z\text{ because }1+\delta>0
\end{align*}

So we can chose a small $r>0$ and a big $R>r$ such that

\[\sup_{\substack{0<\alpha_1,\alpha_2\leq1 \\ z\in(0,r)\cup(R,\infty)}}\left|\frac{1}{z^{3+\delta}}\int_0^1 f^{\alpha_1}\left(\frac{u}{z^2}\right)f^{\alpha_2}\left(\frac{1-u}{z^2}\right)du\right|\leq\varepsilon\]

Moreover, thanks to the uniform convergence (\ref{inequality_uniform}) for any $\eta>0$ we have

\[\frac{1}{z^{3+\delta}}\int_\eta^{1-\eta} f^{\alpha_1}\left(\frac{u}{z^2}\right)f^{\alpha_2}\left(\frac{1-u}{z^2}\right)du\underset{\substack{\alpha_1\rightarrow1^- \\ \alpha_2\rightarrow1^-}}{\longrightarrow}\frac{1}{z^{3+\delta}}\int_\eta^{1-\eta} f\left(\frac{u}{z^2}\right)f\left(\frac{1-u}{z^2}\right)du\text{ uniformly on }[r,R]\]

Now to conclude the desired convergence, we only need a small $\eta>0$ such that \[\sup_{\substack{0<\alpha_1,\alpha_2\leq1 \\ z\in[r,R]}}\left|\int_0^\eta f^{\alpha_1}\left(\frac{u}{z^2}\right)f^{\alpha_2}\left(\frac{1-u}{z^2}\right)du\right|\leq\varepsilon\]

Using (\ref{bounded_infinity}), there is a constant $c>0$ such that $\forall0<\alpha\leq1$, $||f^{\alpha}||_{\infty,[1/2R^2,\infty)}\leq c$.

For any small $\eta>0$ and $z\in[r,R]$,
\begin{align*}
\left|\int_0^\eta f^{\alpha_1}\left(\frac{u}{z^2}\right)f^{\alpha_2}\left(\frac{1-u}{z^2}\right)du\right|&\leq\sqrt{\int_0^\eta f^{\alpha_1}\left(\frac{u}{z^2}\right)^2 du}\sqrt{\int_0^\eta f^{\alpha_2}\left(\frac{1-u}{z^2}\right)^2 du}\\
&\leq z||f^{\alpha_1}||_{L^2(\mathbb{R}_+)}\times ||f^{\alpha_2}||_{\infty,[1/2R^2,\infty)}\sqrt{\eta}\\
&\leq RCc\sqrt{\eta}\\
&\leq\varepsilon\text{ for small }\eta
\end{align*}

Finally, we have proven
\[\frac{1}{z^{3+\delta}}\int_0^1 f^{\alpha_1}\left(\frac{u}{z^2}\right)f^{\alpha_2}\left(\frac{1-u}{z^2}\right)du\underset{\substack{\alpha_1\rightarrow1^- \\ \alpha_2\rightarrow1^-}}{\longrightarrow}\frac{1}{z^{3+\delta}}\int_0^1 f\left(\frac{u}{z^2}\right)f\left(\frac{1-u}{z^2}\right)du\text{ uniformly on }(0,\infty)\]
For the following, notice all equalities are justified thanks to absolute convergence from $\alpha^n$.
\begin{align*}
&\alpha_1\alpha_2\int_0^1 f^{\alpha_1}\left(\frac{u}{z^2}\right)f^{\alpha_2}\left(\frac{1-u}{z^2}\right)du\\
&=\frac{1}{C_\nu^2}\sum_{m,n\geq1}(-1)^{m+n}\alpha_1^n\alpha_2^{m}g_m g_n\exp\left(-\frac{j_m^2}{2z^2}\right)\int_0^1\exp\left(\frac{j_m^2-j_n^2}{2z^2}u\right)du\\
&=\frac{1}{C_\nu^2}\sum_{n\geq1}(\alpha_1\alpha_2)^{n}g_n^2\exp\left(-\frac{j_n^2}{2z^2}\right)+\frac{2z^2}{C_\nu^2}\sum_{\substack{m,n\geq1\\ m\neq n}}(-1)^{m+n}\alpha_1^n\alpha_2^{m}\frac{g_m g_n}{j_m^2-j_n^2}\left(e^{-\frac{j_{n}^2}{2z^2}}-e^{-\frac{j_{m}^2}{2z^2}}\right)
\end{align*}

\[\sum_{n\geq1}(\alpha_1\alpha_2)^{n}g_n^2\exp\left(-\frac{j_n^2}{2z^2}\right)\underset{\substack{\alpha_1\rightarrow1^- \\ \alpha_2\rightarrow1^-}}{\longrightarrow}\sum_{n\geq1}g_n^2\exp\left(-\frac{j_n^2}{2z^2}\right)<\infty\text{ by monotone convergence}\]

\[\int_0^1 f\left(\frac{u}{z^2}\right)f\left(\frac{1-u}{z^2}\right)du=\frac{1}{C_\nu^2}\sum_{n\geq1}g_n^2 e^{-\frac{j_n^2}{2z^2}}+\frac{2z^2}{C_\nu^2}\sum_{\substack{m,n\geq1\\ m\neq n}}(-1)^{m+n}\frac{g_m g_n}{j_m^2-j_n^2}\left(e^{-\frac{j_{n}^2}{2z^2}}-e^{-\frac{j_{m}^2}{2z^2}}\right)\ (\mathcal{A})\]

From that and (\ref{agreement_bessel}) we deduce the first theorem (\ref{density_M_bessel})
\[\frac{\P(M\in dz)}{dz}=\frac{2}{C_\nu z^{3+\delta}}\sum_{n\geq1}g_n^2 e^{-\frac{j_n^2}{2z^2}}+\frac{4}{C_\nu z^{1+\delta}}\sum_{\substack{m,n\geq1\\ m\neq n}}(-1)^{m+n}\frac{g_m g_n}{j_m^2-j_n^2}\left(e^{-\frac{j_{n}^2}{2z^2}}-e^{-\frac{j_{m}^2}{2z^2}}\right)\ (\mathcal{A})\]

Let $U\in(0,1/2)$, $u\in[U,1-U]$ and $\varepsilon>0$.
\begin{eqnarray}
\int_0^\infty \frac{2}{z^{3+\delta}}f\left(\frac{u}{z^2}\right)f\left(\frac{1-u}{z^2}\right)dz &=& 2\int_0^\infty f(x^2 u)f(x^2(1-u))x^{1+\delta}dx\nonumber\\
&=& \int_0^\infty f(xu) f(x(1-u))x^{\frac{\delta}{2}}dx<\infty \label{change_variable}
\end{eqnarray}

Approximation for $x$ near $0$:

Let us take a small $r$, for any $0<\alpha_1,\alpha_2\leq1$
\begin{align*}
\left|\int_0^r f^{\alpha_1}(xu) f^{\alpha_2}(x(1-u))x^{\frac{\delta}{2}}dx\right|&\leq r^{\frac{1+\delta}{2}}\int_0^\infty \frac{1}{\sqrt{x}}|f^{\alpha_1}(xu) f^{\alpha_2}(x(1-u))| dx\\
&\leq r^{\frac{1+\delta}{2}}\sqrt{\int_0^\infty \frac{f^{\alpha_1}(xu)^2}{ \sqrt{x}}dx}\sqrt{\int_0^\infty \frac{f^{\alpha_2}(x(1-u))^2}{ \sqrt{x}}dx}\\
&\leq \frac{C(1/2)^2}{\sqrt[4]{U(1-U)}}r^{\frac{1+\delta}{2}}\\
&\leq \varepsilon\text{ for small }r\text{ because }1+\delta>0
\end{align*}

Approximation for $x$ near $\infty$:

Thanks to (\ref{bounded_infinity}), for any $t$ big enough, $\forall\alpha\in(0,1]$, $|f^\alpha(t)|\leq e^{-\frac{j_0^2}{4}t}$, and for $R$ big enough
\begin{align*}
\left|\int_R^\infty f^{\alpha_1}(xu) f^{\alpha_2}(x(1-u))x^{\frac{\delta}{2}}dx\right|&\leq\int_R^\infty |f^{\alpha_1}(xu)| |f^{\alpha_2}(x(1-u))|x^{\frac{\delta}{2}}dx\\
&\leq\int_R^\infty x^{\frac{\delta}{2}}e^{-\frac{j_0^2}{4}x} dx\\
&\leq\varepsilon\text{ for big }R\text{, by dominated convergence}
\end{align*}

After choosing a small $r$ and a big $R$, \[\sup_{\substack{0<\alpha_1,\alpha_2\leq1 \\ u\in[U,1-U]}}\left|\int_0^r f^{\alpha_1}(xu) f^{\alpha_2}(x(1-u))x^{\frac{\delta}{2}}dx\right|+\left|\int_R^\infty f^{\alpha_1}(xu) f^{\alpha_2}(x(1-u))x^{\frac{\delta}{2}}dx\right|\leq\varepsilon\]
Moreover by (\ref{inequality_uniform}), uniformly on $[U,1-U]$
\[\int_r^R f^{\alpha_1}(xu) f^{\alpha_2}(x(1-u))x^{\frac{\delta}{2}}dx\underset{\substack{\alpha_1\rightarrow1^- \\ \alpha_2\rightarrow1^-}}{\longrightarrow}\int_r^R f(xu) f(x(1-u))x^{\frac{\delta}{2}}dx\]
Eventually,
\[\int_0^\infty f^{\alpha_1}(xu) f^{\alpha_2}(x(1-u))x^{\frac{\delta}{2}}dx\underset{\substack{\alpha_1\rightarrow1^- \\ \alpha_2\rightarrow1^-}}{\longrightarrow}\int_0^\infty f(xu) f(x(1-u))x^{\frac{\delta}{2}}dx\text{ uniformly on }[U,1-U]\]

\begin{align*}
&\alpha_1\alpha_2\int_0^\infty f^{\alpha_1}(xu) f^{\alpha_2}(x(1-u))x^{\frac{\delta}{2}}dx\\
&=\frac{1}{C_\nu^2}\sum_{m,n\geq1}(-1)^{m+n}\alpha_1^n\alpha_2^{m}g_m g_n \int_0^\infty x^{\nu+1}\exp\left(-\frac{j_n^2 u +j_m^2 (1-u)}{2}x\right)dx\\
&=\frac{1}{C_\nu^2}\sum_{m,n\geq1}(-1)^{m+n}\alpha_1^n\alpha_2^{m}g_m g_n \left(\frac{2}{j_n^2 u +j_m^2 (1-u)}\right)^{\nu+2}\int_0^\infty x^{\nu+1}e^{-x}dx\\
&=\frac{2^{\nu+2}\Gamma(\nu+2)}{C_\nu^2}\sum_{m,n\geq1}(-1)^{m+n}\alpha_1^n\alpha_2^{m}\frac{g_m g_n}{[j_n^2 u +j_m^2 (1-u)]^{\nu+2}}
\end{align*}

\[\frac{2^{\nu+2}\Gamma(\nu+2)}{C_\nu^2}=4(\nu+1)\times\frac{2^\nu \Gamma(\nu+1)}{C_\nu^2}=\frac{2\delta}{C_\nu}\]

From that and (\ref{agreement_bessel}), we find the second theorem (\ref{density_rho_bessel})
\[\frac{\P(\rho\in du)}{du}=2\delta\sum_{m,n\geq1}(-1)^{m+n}\frac{g_m g_n}{[j_n^2 u +j_m^2 (1-u)]^{\nu+2}}\ (\mathcal{A})\]
\end{proof}

For the Brownian excursion and the reflected Brownian bridge, there are explicit expressions for $j_n$ and $g_n$, and the formulas can be written
\begin{theorem}Excursion: $\nu=1/2$

\[\frac{\P(M\in dz)}{dz}=\frac{\sqrt{2\pi}\pi^4}{z^6}\sum_{n\geq1}n^4 e^{-\frac{n^2\pi^2}{2z^2}}\]\begin{equation}
\label{density_M_excursion}+\frac{2\sqrt{2\pi}\pi^2}{z^4}\sum_{\substack{m,n\geq1 \\ m\neq n}}(-1)^{m+n}\frac{m^2 n^2}{m^2 -n^2}\left(e^{-\frac{n^2\pi^2}{2z^2}}-e^{-\frac{m^2\pi^2}{2z^2}}\right)\ (\mathcal{A})
\end{equation}

\begin{equation}
\label{density_rho_excursion}
\frac{\P(\rho\in du)}{du}=3\sum_{m,n\geq1}(-1)^{m+n}\frac{m^2 n^2}{[n^2 u+m^2 (1-u)]^{5/2}}\ (\mathcal{A})
\end{equation}
\end{theorem}

\begin{theorem}Reflected bridge: $\nu=-1/2$
\[\frac{\P(M\in dz)}{dz}=\frac{\sqrt{2\pi}\pi^2}{z^4}\sum_{n\geq1}\left(n+\frac{1}{2}\right)^2 e^{-\frac{(n+1/2)^2\pi^2}{2z^2}}\]\begin{equation}
\label{density_M_reflected}
+\frac{2\sqrt{2\pi}}{z^2}\sum_{\substack{m,n\geq1 \\ m\neq n}}(-1)^{m+n}\frac{(m+1/2)(n+1/2)}{(m+1/2)^2 -(n+1/2)^2}\left(e^{-\frac{(n+1/2)^2\pi^2}{2z^2}}-e^{-\frac{(m+1/2)^2\pi^2}{2z^2}}\right)\ (\mathcal{A})\end{equation}

\begin{equation}
\label{density_rho_reflected}
\frac{\P(\rho\in du)}{du}=2\sum_{m,n\geq0}(-1)^{m+n}\frac{(2m+1)(2n+1)}{[(2n+1)^2 u+(2m+1)^2 (1-u)]^{3/2}}\ (\mathcal{A})
\end{equation}
\end{theorem}

\begin{remark}
The formula (\ref{density_M_bessel}) has been found in \cite{MR1701890} (section 11) for $\nu<-\frac{1}{2}$, but with the standard summation. Moreover, in this article, Pitman and Yor deduce formulas on the zeros of the Bessel functions, equating their double series with other expression of the density of $M$ (called the Gikhman-Kiefer Formula), and we will do the same later. Trying to do so with $\rho$ may be a goal.

The use of the Abel summation is not coming from nowhere and is primary. First, (\ref{density_M_bessel}) becomes false with the standard summation because they are a infinity of not-small terms, as said in \cite{MR1701890}. However, Pitman and Yor pointed out the formula could become correct for $\nu=\frac{1}{2}$ (the case of the excursion) when you use the equality $\sum_{n\geq1}(-1)^{n-1}=1/2$, as it happens with Abel summation. Furthermore, in \cite{majumdar}, the authors interpret their numerical calculations with a regularization $\alpha\rightarrow1$.
\end{remark}

The summation in (\ref{density_M_bessel}) is absolutely convergent iff $\nu<-\frac{1}{2}$ while the summation in (\ref{density_rho_bessel}) is never absolutely convergent.

\subsection{Proof of the lemma}

We keep the same notations as before

\begin{align*}
\int_0^\infty f^\alpha(t)^2 t^{q-1}dt&=\int_0^\infty\frac{1}{C_\nu^2}\sum_{m,n\geq1}(-\alpha)^{m+n-2}g_m g_n\exp\left(-\frac{j_m^2+j_n^2}{2}t\right)t^{q-1}dt\\
&=\frac{1}{C_\nu^2}\sum_{m,n\geq1}(-\alpha)^{m+n-2}g_m g_n\int_0^\infty t^{q-1}\exp\left(-\frac{j_m^2+j_n^2}{2}t\right)dt\\
&=\frac{2^q \Gamma(q)}{C_\nu^2}\sum_{m,n\geq1}(-\alpha)^{m+n-2}\frac{g_m g_n}{(j_m^2+j_n^2)^q}\\
&=\frac{2^q \Gamma(q)}{C_\nu^2}\sum_{n\geq2}(-\alpha)^{n-2}u_n\text{ with }u_n=\sum_{k=1}^{n-1}\frac{g_k g_{n-k}}{(j_k^2+j_{n-k}^2)^q}
\end{align*}

$(u_n)$ has a polynomial growth so we only need to show $((-1)^n u_n)$ is Abel summable. The strategy of the proof is to provide an asymptotic expansion of $u_n$ where all the terms are either of the form $\text{constant}\times n^a$ either absolutely summable. For this, we use the existence of asymptotic expansions of this form at any order for $j_n^2$ and $g_n$: $\forall N\in\mathbb{N}^*,$
\begin{align*}
\frac{j_n^2}{(\pi n)^2}&=1+\sum_{i=1}^{N-1} \text{constant(i)}\times n^{-i}+O(n^{-N})\\
\sqrt{\frac{2}{\pi}}\frac{g_n}{(\pi n)^{\nu+3/2}}&=1+\sum_{i=1}^{N-1}\text{constant}'(i)\times n^{-i}+O(n^{-N})
\end{align*}

It comes directly (after some calculus) from Hankel's expansion of $J_{\nu+1}(z)$ for large $z$ and McMahon's expansion of $j_n$ for large $n$ (see \cite{asymptotics} pages 237-242 and 247-248). Then, we will conclude thanks to the following fact
\begin{proposition}
\label{eta}
$\forall s\in\mathbb{C}$, $\displaystyle{\sum_{n\geq1}\frac{(-1)^{n-1}}{n^s}x^{n-1}}$ is absolutely convergent for $x\in\mathbb{C}$, $|x|<1$ and can be continuously extended to $1$. In particular, $((-1)^n n^{-s})$ is Abel summable.
\end{proposition}

\begin{remark}
In fact, $\displaystyle{\eta(s)=\sum_{n\geq1}(-1)^{n-1} n^{-s}\ (\mathcal{A})}$ defines an entire function. Since the series is absolutely convergent for $\Re s>1$, we can easily verify the relation $\eta(s)=(1-2^{1-s})\zeta(s)$ where $\zeta$ is the Riemann zeta function. Then, we can use the relation to extend $\zeta$ to $\mathbb{C}$ except to the points $s_n=1+2i n\pi/\ln(2)$, $n\in\mathbb{Z}$. More work is required to show $\eta(s_n)=0$ for $n\neq0$ and to give an analytic definition of $\zeta$ on $\mathbb{C}\backslash\{1\}$. Moreover, the well-known value of the alternating harmonic series gives $\eta(1)=\ln(2)$ and so $(s-1)\zeta(s)\underset{s\rightarrow1}{\longrightarrow}1$. While $(n)$ is not Abel summable, it is not difficult to prove $\eta(-1)=1/4$ to assign to $\displaystyle{\sum_{n\geq1}n}$ the famous value $\zeta(-1)=-1/12$. See \cite{Kanemitsu2000} for more details and a proof of the proposition. The proof of the Abel summability of $((-1)^n n^{-s})$ is also done in \cite{hardy2000divergent}, Section 6.10. \cite{10.2307/3647831} solves the problem of the vanishing of $\eta$ at the $s_n$.
\end{remark}

We use the symbol $\preceq$ when an inequality is true up to a constant independent from $k$ or $n$. We denote \[x(k,n)=\frac{j_k^2-\pi^2 k^2+j_{n-k}^2-\pi^2(n-k)^2}{\pi^2 k^2+\pi^2(n-k)^2}\]
\[\frac{1}{j_k^2+j_{n-k}^2}=\frac{1}{\pi^2(k^2+(n-k)^2)}\frac{1}{1+x(k,n)}\]
\[|x(k,n)|\preceq\frac{k+(n-k)}{k^2+(n-k)^2}\preceq\frac{n}{n^2}\preceq\frac{1}{n}\]
We have a polynomial whom does not depend on the problem such that \[\frac{1}{(1+x)^q}-P(x)=O(x^{\delta+4})\]
\begin{align*}
\left|u_n-\sum_{k=1}^{n-1}\frac{g_k g_{n-k}}{\pi^{2q}(k^2+(n-k)^2)^q}P(x(k,n))\right|&\preceq\sum_{k=1}^{n-1}\frac{|g_k g_{n-k}|}{\pi^{2q}(k^2+(n-k)^2)^q}x(k,n)^{\delta+4}\\
&\preceq\sum_{k=1}^{n-1}\frac{k^{\nu+3/2} (n-k)^{\nu+3/2}}{\pi^{2q}(k^2+(n-k)^2)^q}\frac{1}{n^{\delta+4}}\\
&\preceq\sum_{k=1}^{n-1}\frac{n^{2\nu+3}}{n^{\delta+4}}\\
&\preceq\frac{1}{n^2}
\end{align*}
By linearity, we only need the Abel summability of $\displaystyle{(-1)^n\sum_{k=1}^{n-1}\frac{g_k g_{n-k}}{(k^2+(n-k)^2)^q}x(k,n)^d}$ where $d\in\mathbb{N}$. We write the asymptotic expansion for $j_k^2$ and $g_k$ with order $\displaystyle{O\left(\frac{1}{k^{d+\nu+3/2+2}}\right)}$ and after expanding $\displaystyle{\frac{g_k g_{n-k}}{(k^2+(n-k)^2)^q}x(k,n)^d}$, without precising the constant coefficients, we have a finite number of terms of the form $\displaystyle{\frac{\text{function}(k)\times\hat{\text{function}}(n-k)}{(k^2+(n-k)^2)^{q+d}}}$. The functions do not depend on $k$ or $n$ and can be either explicit, with the form $\text{function}(k)=k^a$ where $a\leq d+\nu+3/2$, either implicit, with the form $\displaystyle{\text{function}(k)=\frac{\mu(k)}{k^{2}}}$ where $\mu$ is a bounded function.

Studying each type of terms, we want to show $\displaystyle{\lim_{\alpha\rightarrow1}\sum_{m,n\geq1}\alpha^{m-1}\alpha^{n-1}(-1)^{m+n}\frac{\text{function}(m)\times\hat{\text{function}}(n) }{(m^2+n^2)^{q+d}}}$ exists.

	\begin{lemma}
	Let $q>0$. If $\displaystyle{f(x,y)=\sum_{m,n\geq1}x^{m-1}y^{n-1}u_{m,n}}$ absolutely converges  for $|x|,|y|<1$ and $x,y\in\mathbb{C}$ and if it can be continuously extended to $(1,1)$ then $\displaystyle{g(x,y)=\sum_{m,n\geq1}x^{m-1}y^{n-1}\frac{u_{m,n}}{(m^2+n^2)^q}}$ satisfies the same properties.
	\end{lemma}
	
	It is simply because for $|x|,|y|<1$ we have the identity
	\[g(x,y)=\frac{2^q}{\Gamma(q)^2}\int_0^1\int_0^1 |\ln t|^{q-1}|\ln s|^{q-1} ts f(t^{1+i}s^{1-i}x,t^{1-i}s^{1+i}y)dt ds\]
	that extends continuously $g$ to $(1,1)$ (by dominated convergence).
	
	Since $q>0$ and $d\geq0$, this lemma allows us to just do the case $q+d=0$ where the double series can be factorized. We just have to show $\displaystyle{\sum_{n\geq1}x^{n-1}(-1)^{n-1}\text{function}(n)}$ absolutely converges for $|x|<1$ and can be continuously extended to $1$. If the function is implicit, function$(k)=\mu(k)/k^2$, it is true by dominated convergence. If the function is explicit, function$(k)=k^a$, it is the proposition 2.
	
Our proof works because $g_n$ and $j_n$ have a very regular growth. They can be asymptotically expanded at any order, with only sequences $k_n$ such that $((-1)^n k_n)$ are Abel summable. We have also used the convergence of $j_n/n^2$.

\subsection{Complements for double series formulas for Bessel bridges}

\begin{theorem}
If $\varphi$ a bounded measurable function then \[\int_0^1\int_0^\infty \varphi(z,u)\times \frac{2C_\nu}{z^{3+\delta}} f^{\alpha_1}\left(\frac{u}{z^2}\right) f^{\alpha_2}\left(\frac{1-u}{z^2}\right)dzdu \underset{\substack{\alpha_1\rightarrow1^- \\ \alpha_2\rightarrow1^-}}{\longrightarrow}\E[\varphi(M,\rho)]\]

In particular, the Abel summation in (\ref{density_M_bessel}) and (\ref{density_rho_bessel}) commute with the integration when we compute the expected value of a bounded function of $M$ or $\rho$.
\end{theorem}

\begin{proof} We want to show \[\int_0^1\int_0^\infty \varphi(u,x)\times f^\alpha(xu) f^\alpha(x(1-u))x^{\frac{\delta}{2}}dx du \underset{\substack{\alpha_1\rightarrow1^- \\ \alpha_2\rightarrow1^-}}{\longrightarrow}\int_0^1\int_0^\infty \varphi(u,x)\times  f(xu) f(x(1-u))x^{\frac{\delta}{2}}dx du\]Thanks to the uniform convergence of $f^\alpha$ on any segment of $(0,\infty)$ (see (\ref{inequality_uniform})), we just have to dominate the behavior of $|f^{\alpha_1}(xu) f^{\alpha_2}(x(1-u))x^{\frac{\delta}{2}}|$ at the boundaries of $(0,1)\times(0,\infty)$.

$(0,1)\times[R,\infty)$:

\begin{align*}
\int_0^1\int_R^\infty |f^{\alpha_1}(xu) f^{\alpha_2}(x(1-u))x^{\frac{\delta}{2}}|dx du&=2\int_0^{\frac{1}{2}}\int_R^\infty |f^{\alpha_1}(xu) f^{\alpha_2}(x(1-u))x^{\frac{\delta}{2}}|dx du\\
&\leq 2\int_0^\frac{1}{2}\int_R^\infty x^{\frac{\delta}{2}}e^{-\frac{j_0^2}{4}x}dx\text{ for big }R\\
&\leq\varepsilon\text{ for bigger }R\text{, by dominated convergence}
\end{align*}

$[0,1]\times[0,r]$:

\[\int_0^1\int_0^r |f^{\alpha_1}(xu) f^{\alpha_2}(x(1-u))x^{\frac{\delta}{2}}|dx du\leq r^{\frac{\delta}{2}}\int_0^1 \frac{C^2}{\sqrt{u(1-u)}}du\leq\varepsilon\text{ for small }r\]

We fix $r>0$ small enough and $R>r$ big enough. Since $||f^\alpha||_{\infty,[r/2,\infty)}\leq c<\infty$ for all $0<\alpha\leq1$ by (\ref{bounded_infinity}), for a small $\eta>0$

$[0,\eta]\times[r,R]$:
\begin{align*}
\int_0^\eta\int_r^R |f^{\alpha_1}(xu) f^{\alpha_2}(x(1-u))x^{\frac{\delta}{2}}|dx du&\leq\int_r^R x^{\frac{\delta}{2}}\int_0^\eta |f^{\alpha_1}(xu) f^{\alpha_2}(x(1-u))|du dx\\
&\leq \sqrt{\eta}cC\int_r^R x^{\frac{\delta-1}{2}} dx\\
&\leq\varepsilon\text{ for small }\eta
\end{align*}

\end{proof}

\begin{conjecture}
$\forall n\in\mathbb{N}^*,\ \forall\delta>0$,
\begin{equation}
\label{conj}\frac{\delta}{4}\frac{j_{\nu,n}^{\nu-1}}{J_{\nu+1}(j_{\nu,n})}=\sum_{\substack{m\geq1 \\ m\neq n}}\frac{1}{j_{\nu,n}^2-j_{\nu,m}^2}\times\frac{j_{\nu,m}^{\nu+1}}{J_{\nu+1}(j_{\nu,m})}\ (\mathcal{A})
\end{equation}
\end{conjecture}

\begin{remark}
This statement was proven in \cite{MR1701890} for $-1<\nu\leq-1/2$ with absolute convergence. Pitman and Yor have also remarked it becomes true for $\nu=1/2$ if $\sum_{n\geq0}(-1)^2=1/2$, that is the case with the Abel's summation. We would like to adapt their method for the Abel summation.
\end{remark}

Pitman and Yor have derived in \cite{MR1701890} an other expression for the density of $M$ called the Gikhman-Kiefer formula
\begin{equation}
\label{gikhman}
\frac{\P(M_\delta\in dz)}{dz}=\frac{2}{C_\nu z^{\delta}}\sum_{n\geq1}\frac{j_{\nu,n}^{2\nu}}{J_{\nu+1}^2(j_{\nu,n})}\left(\frac{j_{\nu,n}^2}{z^3}-\frac{\delta}{z}\right)e^{-\frac{j_{\nu,n}^2}{2z^2}}
\end{equation}

Combining this with (\ref{density_M_bessel}), and after a change a variable $t=1/z^2$ we obtain 
\[\forall t>0,\ -\frac{\delta}{2}\sum_{n\geq1}\frac{j_{\nu,n}^{2\nu}}{J_{\nu+1}^2(j_{\nu,n})}e^{-\frac{j_{\nu,n}^2}{2}t}=\sum_{\substack{m,n\geq 1\\ m\neq n}}\frac{j_{\nu,m}^{\nu+1}}{J_{\nu+1}(j_{\nu,m})}\frac{j_{\nu,n}^{\nu+1}}{J_{\nu+1}(j_{\nu,n})}\frac{e^{-\frac{j_{\nu,n}^2}{2}t}-e^{-\frac{j_{\nu,m}^2}{2}t}}{j_{\nu,m}^2-j_{\nu,n}^2}\ (\mathcal{A})\]

Moreover, we have the equality

\[\sum_{\substack{m,n\geq 1\\ m\neq n}}\alpha^{m+n}\frac{j_{\nu,m}^{\nu+1}}{J_{\nu+1}(j_{\nu,m})}\frac{j_{\nu,n}^{\nu+1}}{J_{\nu+1}(j_{\nu,n})}\frac{e^{-\frac{j_{\nu,n}^2}{2}t}-e^{-\frac{j_{\nu,m}^2}{2}t}}{j_{\nu,m}^2-j_{\nu,n}^2}=-2\sum_{n\geq1}\alpha^n\frac{j_{\nu,n}^{\nu+1}}{J_{\nu+1}(j_{\nu,n})}e^{-\frac{j_{\nu,n}^2}{2}t}\sum_{\substack{m\geq1 \\ m\neq n}}\frac{\alpha^m}{j_{\nu,n}^2-j_{\nu,m}^2}\frac{j_{\nu,m}^{\nu+1}}{J_{\nu+1}(j_{\nu,m})}\]

We would like to identify the coefficients of $e^{-\frac{j_{\nu,n}^2}{2}t}$ but we lack a stronger convergence to do this. We can also say that Pitman and Yor showed how (\ref{conj}) can be deduced from the Mittag-Leffler expansion (true when $-1<\nu<-12$)

\[\frac{x^\nu}{J_\nu(x)}=2\sum_{m\geq1}\frac{1}{j_{\nu,m}^2-x^2}\times\frac{j_{\nu,m}^{\nu+1}}{J_{\nu+1}(j_{\nu,m})}\]

\subsection{A key point of the proof: the series expansion of $f$}

Our proof of (\ref{density_M_bessel}) and (\ref{density_rho_bessel}) lies on 3 (maybe 4) key points: the \textbf{agreement formula}, the \textbf{series expansion} of $f$, and the \textbf{lemma} \ref{lemma} (and maybe the scaling property). The agreement formula was already presented in section 2, and the proof of lemma only depends on the series expansion of $f$ and their regularity. Let us see where does this expansion come from, and if it could exist for other processes than Bessel processes.

In \cite{kent}, Kent applied Sturm-Liouville theory of second order differential equations and complex analysis to study the distribution of diffusion hitting times. For a regular diffusion on $[r_0,r_1]$, the generator $A$ is a generalized second order linear differential operator. If $r_0<a<b<r_1$, $\lambda>0$, there is a unique solution $v(x)$ up to a multiplicative constant of $Av=\lambda v$ together with a boundary condition at $r_0$ (depending on the boundary behavior of the scale function and the speed measure). Moreover, we have the relation $\E_a(\exp(-\lambda T_b))=v(a)/v(b)$.

If $r_0$ is not a natural boundary (the diffusion cannot start at $r_0$ nor reach $r_0$ in finite time with positive probability, for example $r_0=-\infty$), we even have a unique joint solution $u(x,\lambda)$ satisfying some initial and boundary conditions at $r_0$ and this solution is entire in $\lambda$. Plus, its zeros corresponds to the eigenvalues of a Sturm-Liouville problem defined by $A$. Thanks to the Sturm-Liouville theory, Kent concludes the zeros $-\lambda_{x,n}$ are negative, simple and go to $-\infty$ pretty fast. So, he could apply Hadamard's factorization theorem to express $\E_a(\exp(-\lambda T_b))$ as a infinite product indexed by the zeros.

To recap what Kent have done in its article, he shows that under some mild regularity conditions, and if $r_0$ is not a natural boundary, then the $\P_a$ density of $T_b$, when $r_0\leq a<b<r_1$, is
\[f_{ab}(t)=\sum d_k e^{-\lambda_{k,b}t}\text{ where }d_k=-u(a,\lambda_{k,b})/u'(b,\lambda_{k,b})\]
and $\forall\varepsilon>0,\ d_k=O(e^{\lambda_{k,b}\varepsilon})$, and $\sum\lambda_{k,b}^{-1}<\infty$.

Bessel processes satisfy all those hypothesis and (\ref{hitting_bessel}) comes from this method. However, it is not the case for many diffusion processes. For example, $r_0=-\infty$ for skew Brownian motions or the Ornstein-Uhlenbeck process. While such series expansion exists for the Ornstein-Uhlenbeck process anyway (see \cite{ou} for a complete study), it is impossible for skew Brownian motions. A simple way to see it is to observe such formula necessarily implies $\E_a(T_b)<\infty$, that is not the case for skew Brownian motions.

\[\E_a(T_b)\leq1+\E_a(T_b1[T_b>1])\leq1+\frac{4}{\lambda_{1,k}}\sum \frac{d_k}{\lambda_{k,b}}e^{-\lambda_{k,b}/2}<\infty\]

\subsection{How could we generalize those formulae ?}

Let us suppose the 3 key points of our proof: the \textbf{agreement formula} of course, but also the \textbf{series expansion of $f$} and the \textbf{lemma} \ref{lemma}. We verify which complementary hypothesis we need to make our proof work. More precisely, we suppose

\begin{eqnarray}
\forall z>0,\ u\in[0,U],&\ \P(M\in dz,\rho\in du) =c \phi_1(z,u)\phi_2(z,U-u)s'(z)dzdu\\
\forall z>0,\ t>0,\ i=1\text{ or }2,&\ \phi_i(z,t) = \sum_{n\geq1}d_{i,n}(z)e^{-\lambda_{i,n}(z)t}\\
\exists C>0,\forall z>0,\ \forall\alpha\in(0,1),&\ ||\phi_1^\alpha(z,.)||_{L^2(\mathbb{R_+})}+||\phi_2^\alpha(z,.)||_{L^2(\mathbb{R_+})}\leq C
\end{eqnarray}

where $\displaystyle{\phi_i^\alpha(z,t) = \sum_{n\geq1}\alpha^n d_{i,n}(z)e^{-\lambda_{i,n}(z)t}}$

In the case of a bridge, $\phi_1(z,u)=f_{xz}(u)$ and $\phi_2(z,u)=f_{yz}(u)$ where $x$ and $y$ are some fixed points. To have an expression of the form of (\ref{density_M_bessel}) for $\P(M\in dz)$ (without asking the uniform convergence), we need

\begin{align*}
&0<\lambda_{i,n}(z)\underset{n\rightarrow\infty}{\longrightarrow}+\infty\\
&\sum_{n\geq1}|d_{i,n}(z)| n e^{-\lambda_{i,n}(z)t}<\infty\\
&\sum_{m,n\geq1}\alpha^{m+n}|d_{1,m}(z)d_{2,n}(z)|e^{-\lambda_{1,m}(z)\wedge\lambda_{2,n}(z)U}\leq\sum_{m,n\geq1}\alpha^{m+n}|d_{1,m}(z) d_{2,n}(z)|<\infty\\
\end{align*}

The conditions depends only on the coefficients $d_n$ and $\lambda_n$, and we do not need any assumptions on $s'$ or the Brownian scaling, because we integrate the agreement formula with respect for $u$. They are satisfied for example when

\[0<\frac{\lambda_n}{n}\underset{n\rightarrow\infty}{\longrightarrow}+\infty\text{ and }\forall\varepsilon>0,\ d_n=O(e^{\varepsilon n})\]

that is the case for the Bessel bridges.

Then, we can use the same proof as before to derive

\[\frac{\P(M\in dz)}{dz}=c s'(z)\sum_{m,n\geq1}d_{1,m}(z) d_{2,n}(z)\frac{e^{-\lambda_{2,n}(z)U}-e^{-\lambda_{1,m}(z)U}}{\lambda_{1,m}(z)-\lambda_{2,n}(z)}\ (\mathcal{A})\]

with the convention $\displaystyle{\frac{e^{-\lambda_{2,n}(z)U}-e^{-\lambda_{1,m}(z)U}}{\lambda_{1,m}(z)-\lambda_{2,n}(z)}=e^{-\lambda_{2,n}(z)U}}$ if $\lambda_{1,m}(z)=\lambda_{2,n}(z)$

It seems rather easy to derive such expressions for more processes. If you have the argument formula, and could apply Kent's theory to obtain some series expansions, you just have to study the asymptotic behavior of $d_n$ and $\lambda_n$ to verify the lemma \ref{lemma} and the above conditions. Even if it does not work well an other possibility of generalization is to replace $\alpha^n$ by an other regularizing sequence in the definition of $\phi^\alpha$. For example, if we take $e^{-\lambda_n\varepsilon}$, $\phi^\varepsilon(t)=\phi(t+\varepsilon)$.

However, to derive an expression of the form of (\ref{density_rho_bessel}) for $\P(\rho\in du)$, our proof also asks to have a \textbf{scaling property} such that the agreement formula could be written

\[\forall z>0,\ u\in[0,U],\ \P(M\in dz,\rho\in du) =c \phi_1\left(\frac{u}{w^{-1}(z)}\right)\phi_2\left(\frac{U-u}{w^{-1}(z)}\right)\frac{s'(z)}{w^{-1}(z)^2}dzdu\]

where $w$ is a continuous strictly increasing function on $[0,\infty)$ such that $w(0)=0$ and $w(\infty)=\infty$ (with the Brownian scaling $w(z)=\sqrt{z}$). It is necessary because we want to integrate with respect to $z$ and so remove the dependance in $z$ of $d_n(z)$ and $\lambda_n(z)$. In the beginning of this paper, such simplification was possible because we were in the case of a bridge from $0$ to $0$, with a \textbf{self-similarity property}. More precisely, $\phi_1(z,t)=\phi_2(z,t)=f_{0z}(t)$ and $(X_{at})_{t\geq0}\ed(w(a)X_t)_{t\geq0}$, so $f_{0z}(t)=f_{01}(t/w^{-1}(z))/w^{-1}(z)$. Note this does not work for other bridges from $x$ to $y$, when $x\neq0$ or $y\neq0$. We precise the definition of a self-similar process.

\begin{definition}
A stochastic process $(X_t)_{t\geq0}$, continuous in probability with $X_0=0$, is self-similar (or semi-stable) when there is a function $w$ on $(0,\infty)$ such that $\forall a>0,\ (X_{at})_{t\geq0}\ed(w(a)X_t)_{t\geq0}$.
\end{definition}

It is easy to see that if $X$ has a non-degenerate distribution then $w(a)=a^q$ for some $q\geq0$. We call $q$ the order of $X$.

\begin{remark}
Bessel processes and skew Brownian motions are self-similar of order $1/2$.
\end{remark}

In \cite{lamperti}, Lamperti classified the self-similar Markov processes and in particular, the self-similar diffusion. His results implies any non-degenerate semi-stable diffusion on $[0,\infty)$ such that $0$ is not absorbing (with other words $\P_0(T_1<\infty)>0$) are of the form

\begin{equation}
\label{lamperti_form}
X=cR^{2q}\text{ with }c>0\text{ and }R\ed \text{BES}(\delta)\text{ for some }\delta>0
\end{equation}

Moreover, the left boundary of a self-similar diffusion can only be $-\infty$ or $0$, so the series expansion theorem of Kent (see previous Section) could only be applied if the diffusion is on $[0,\infty)$. But if $X$ is also regular on $\mathbb{R}$, it is possible to show $\E_0(T_1)=\infty$ thanks to the above description. So, we cannot hope for a series expansion like (\ref{hitting_bessel}) if the left boundary if $-\infty$.

In short, to be able to reproduce the same proof to find an expression of the density of the argmax like (\ref{density_rho_bessel}) for an $X$-bridge from $x$ to $y$, we need $X$ to be a self-similar diffusion on $[0,\infty)$, and $x=y=0$. But, in this case, the density of $\rho$ can simply be deduced from (\ref{density_rho_bessel}) and (\ref{lamperti_form}). Thus, our method to prove the formula (\ref{density_rho_bessel}) can not be generalized to essentially different diffusion bridges without being modified.

Nevertheless, one has to be aware they are the conditions we used in our proof and not necessary conditions. Some expressions like (\ref{density_rho_bessel}) may be found for other processes by using different series expansion or simplifications of the agreement formula. It is what happens for the next example of skew Brownian bridges.

\section{Marginal densities for skew Brownian bridges}
Recently, in \cite{MR3012092}, Appuhamillage and Sheldon derived a series expansion for $f_\beta$, the density of $T$ for the case of a skew Brownian motion such that $P(X_1>0)=\beta\in(0,1)$.

\begin{equation}
\label{hitting_skew}
f_\beta(t)=\frac{2\beta}{\sqrt{2\pi t^3}}\sum_{n\geq1}(1-2\beta)^{n-1}(2n-1)\exp\left(-\frac{(2n-1)^2}{2t}\right)
\end{equation}

Notice this expression is different from the one for the Bessel process. Although (\ref{hitting_bessel}) comes from the Sturm-Liouville theory and the paper of Kent (see Section 4.4 and \cite{kent}), (\ref{hitting_skew}) comes from a conditioning on the signs of the excursions of the skew Brownian motion (details in \cite{MR3012092}). The two methods are essentially different and cannot be applied to the other family of processes. However, it is possible to use (\ref{hitting_skew}) instead of (\ref{hitting_bessel}) with the agreement formula to derive an expression for the density of $\rho$ looking like (\ref{density_rho_bessel}).

We also recall the agreement formula (\ref{agreement_skew}) for a (standard) skew Brownian bridge: for $z>0$ and $0<u<1$
\[\P(M\in dz,\rho\in du)=\frac{\sqrt{2\pi}}{\beta z^4}f_\beta\left(\frac{u}{z^2}\right)f_\beta\left(\frac{1-u}{z^2}\right)dz du\]

\begin{theorem}
The density of the argmax of a standard skew Brownian bridge is equal to
\begin{equation}
\label{density_rho_skew}
\frac{\P(\rho\in du)}{du}=2\beta\sum_{m,n\geq1}(1-2\beta)^{m+n-2}\frac{(2m-1)(2n-1)}{[(2n-1)^2 u+(2m-1)^2 (1-u)]^{3/2}}
\end{equation}
\end{theorem}

\begin{proof} The following equalities are justified by absolute convergence because $|1-2\beta|<1$.
\begin{align*}
\frac{\P(\rho\in du)}{du}&=\frac{\sqrt{2\pi}}{\beta}\times\frac{(2\beta)^2}{2\pi}\int_0^\infty\frac{1}{z^4}\times\frac{z^6}{[u(1-u)]^{3/2}}\sum_{m,n\geq1}(1-2\beta)^{m+n-2}(2m-1)(2n-1)\\
& \times\exp\left(-\frac{(2n-1)^2 u +(2m-1)^2 (1-u)}{2u(1-u)}z^2\right)dz\\
&=\frac{4\beta}{\sqrt{2\pi}}[u(1-u)]^{-3/2}\sum_{m,n\geq1}(1-2\beta)^{m+n-2}(2m-1)(2n-1)\\
&\times\int_0^\infty z^2\exp\left(-\frac{(2n-1)^2 u +(2m-1)^2 (1-u)}{2u(1-u)}z^2\right)dz\\
&=\frac{4\beta}{\sqrt{2\pi}}[u(1-u)]^{-3/2}\sum_{m,n\geq1}(1-2\beta)^{m+n-2}(2m-1)(2n-1)\\
&\times\left[\frac{2u(1-u)}{(2n-1)^2 u +(2m-1)^2 (1-u)}\right]^{3/2}\int_0^\infty z^2e^{-z^2}dz\\
&=\frac{2\beta\times4}{\sqrt{\pi}}\times\frac{\sqrt{\pi}}{4}\sum_{m,n\geq1}(1-2\beta)^{m+n-2}\frac{(2m-1)(2n-1)}{[(2n-1)^2 u +(2m-1)^2 (1-u)]^{3/2}}
\end{align*}
\end{proof}

For $\beta=1/2$, the case of the (non-reflected) Brownian bridge, (\ref{hitting_skew}) and (\ref{density_rho_skew}) simply become the classical formulas
\[f_{1/2}(t)=\frac{1}{\sqrt{2\pi t^3}}\exp\left(-\frac{1}{2t}\right)\]
\[\frac{\P(\rho\in du)}{du}=1\]

If we denote $\alpha=2\beta-1<1$, $\alpha\longrightarrow1$ when $\beta\longrightarrow1$, or in other words when skew Brownian motion($\beta$) converges in law to the reflected Brownian motion (Bessel process of dimension $\delta=1$). Remarkably, the convergence $\alpha\longrightarrow1$ in (\ref{density_rho_skew}) is exactly the convergence in the definition of the Abel summation used in (\ref{density_rho_reflected}) when $\alpha_1=\alpha_2=\alpha$. (\ref{density_rho_reflected}) is almost a direct consequence of (\ref{density_rho_skew}).
$\newline$

It shows the strong links between (\ref{density_rho_skew}) and (\ref{density_rho_bessel}), even if there is not any obvious connections between (\ref{hitting_skew}) and (\ref{hitting_bessel}). It is very interesting and strange, mainly because the same thing does not happen for $M$.

\begin{theorem}
The density of the max of a standard skew Brownian bridge is equal to
\begin{equation}
\label{density_M_skew}
\frac{\P(M\in dz)}{dz}=8\beta z\sum_{k\geq1}(1-2\beta)^{k-1}k^2e^{-2k^2 z^2}
\end{equation}
\end{theorem}

\begin{proof}
First, we recall the equality for $a,b>0$ (coming from the stability of the laws of Levy)
\[\int_0^1\sqrt{\frac{a}{2\pi}}\frac{e^{-\frac{a}{2u}}}{u^{3/2}}\times\sqrt{\frac{b}{2\pi}}\frac{e^{-\frac{b}{2(1-u)}}}{(1-u)^{3/2}}du=\frac{\sqrt{a}+\sqrt{b}}{\sqrt{2\pi}}e^{-\frac{(\sqrt{a}+\sqrt{b})^2}{2}}\]
By absolute convergence,
\begin{align*}
\frac{\P(M\in dz)}{dz}&=\frac{\sqrt{2\pi}}{\beta z^4}\sum_{m,n\geq1}\frac{(2\beta)^2}{2\pi}(1-2\beta)^{m+n-2}(2m-1)(2n-1)z^6\\
&\times\int_0^1\frac{1}{[u(1-u)]^{3/2}}\exp\left(-\frac{(2n-1)^2z^2}{2u}-\frac{(2m-1)^2z^2}{2(1-u)}\right)du\\
&=\frac{4\beta}{\sqrt{2\pi}}z^2\sum_{m,n\geq1}(1-2\beta)^{m+n-2}(2m-1)(2n-1)\\
&\times \frac{\sqrt{2\pi}}{z}\frac{(2m-1)+(2n-1)}{(2m-1)(2n-1)}\exp\left(-\frac{((2m-1)+(2n-1))^2}{2}z^2\right)\\
&=8\beta z\sum_{m,n\geq1}(1-2\beta)^{m+n-2}(m+n-1)e^{-2(m+n-1)^2z^2}\\
&=8\beta z\sum_{k\geq1}(1-2\beta)^{k-1}k^2e^{-2k^2 z^2}
\end{align*}
\end{proof}

For $\beta=1/2$, the case of the (non-reflected) Brownian bridge, (\ref{density_M_skew}) is again a classical formula
\[\frac{\P(M\in dz)}{dz}=4ze^{-2z^2}\]
For $\beta=1$, the absolute convergence also holds and (\ref{density_M_skew}) becomes the well known formula used for the Kolmogorov-Smirnov test
\[\frac{\P(M\in dz)}{dz}=8 z\sum_{k\geq1}(-1)^{k-1}k^2e^{-2k^2 z^2}\]
This formula belongs to the family of the Gikhman-Kiefer formulae (\ref{gikhman}) for all Bessel bridges of dimension $\delta>0$. More precisely, it is the case $\delta=1$, after using the functional relation for Jacobi's theta function (see \cite{MR1701890} for more details). The links between this formula and (\ref{density_M_skew}) with (\ref{density_M_reflected}) and (\ref{density_M_bessel}) are not obvious at all.

\section{Marginal densities for generalized Bessel meanders}

To use the self-similarity, if the end point of the process is fixed it only can be $0$. However, if the end point has a continuous distribution, we can obtain different processes.

Let $k=2(\mu+1),\delta=2(\nu+1)>0$, the $k$-generalized Bessel meander of dimension $\delta$ is a continuous stochastic process $X$ on $[0,1]$ such that $\displaystyle{\P(X_1\in dy)=\frac{2}{2^{k/2}\Gamma(k/2)}y^{k-1}e^{-\frac{y^2}{2}}}$ and conditionally given $X_1=y$, $X$ is a Bessel bridge of dimension $\delta$, from $0$ to $y$, and of length $1$.

The $k$-generalized Bessel meander of dimension $\delta$ gives a generalization of several processes. When $k=2$, it is the classical Bessel meander of dimension $\delta$. When $\delta=3$, it is the $k$-generalized Brownian meander (seen in Section 3.3 for $k$ an integer). When $k=2$ and $\delta=3$, it is the classical Brownian meander. When $k=\delta$, it is the Bessel process of dimension $\delta$ on $[0,1]$. See \cite{ucb.b1802599320070101} for details and background

We use the notations of Sections 3.1 and 4.1.  $C_\nu=2^\nu\Gamma(\nu+1)$, $f_{x1}$ the density of the first hitting time of $1$ for the Bessel process of dimension $\delta$ started at $x$, and $f=f_{01}$. From \cite{kent}, we have the expressions

\[f(t)=\frac{1}{C_\nu}\sum_{n\geq1}\frac{j_{\nu,n}^{\nu+1}}{J_{\nu+1}(j_{\nu,n})}e^{-\frac{j_{\nu,n}^2}{2}t}=\frac{1}{C_\nu}\sum_{n\geq1}(-1)^{n-1}g_{\nu,n}e^{-\frac{j_{\nu,n}^2}{2}t}\]

\[f_{x1}(t)=\sum_{n\geq1}(-1)^{n-1}g_{\nu,n}\frac{J_\nu(xj_{\nu,n})}{(xj_{\nu,n})^\nu}e^{-\frac{j_{\nu,n}^2}{2}t}\text{ for }x>0\]

We could forget the indexes $\nu$. In this case the agreement formula could be written

\begin{equation}
\label{agreement_generalized_meander}
\P(M\in dz,\rho\in du)=\frac{2C_\nu}{kC_\mu}\times\frac{1}{z^{3+(\delta-k)}}f\left(\frac{u}{z^2}\right)\phi_k\left(\frac{1-u}{z^2}\right)dz du
\end{equation}

\begin{align*}
\text{ where }\phi_k(t)&=\int_0^1 k x^{k-1}f_{x1}(t)dx\geq0\\
&=\sum_{n\geq1}(-1)^{n-1}\left(\int_0^1k x^{k-1}\frac{J_\nu(xj_n)}{(xj_n)^\nu}dx\right)g_n e^{-\frac{j_n^2}{2}t}\\
&=\sum_{n\geq1}(-1)^{n-1}\left(\int_0^{j_n}k x^{k-1}\frac{J_\nu(x)}{x^\nu}dx\right)\frac{g_n}{j_n^{k}} e^{-\frac{j_n^2}{2}t}
\end{align*}

$\phi_k$ is the density of some random time $S_k\geq0$ because$\displaystyle{\int_0^\infty \phi_k(t)dt=1}$.

For $k=2$,
\begin{align*}
\int_0^{j_n} x^{1-\nu}J_{\nu}(x)dx&=[-x^{1-\nu}J_{\nu-1}(x)]_0^{j_n}\\
&=\frac{1}{2^{\nu-1}\Gamma(\nu-1+1)}-j_n^{1-\nu}J_{\nu-1}(j_n)\\
&=\frac{1}{C_{\nu-1}}+j_n^{1-\nu}J_{\nu+1}(j_n)
\end{align*}

\[\phi_2(t)=2\sum_{n\geq1}e^{-\frac{j_n^2}{2}t}+\frac{2}{C_{\nu-1}}\sum_{n\geq1}\frac{j_n^{\nu-1}}{J_{\nu+1}(j_n)}e^{-\frac{j_n^2}{2}t}\]

When $\delta=3$ and $k=2$, case of the Brownian meander, 

\[\phi_2(t)=2\sum_{n\geq1}(1+(-1)^{n-1})e^{-\frac{(\pi n)^2}{2}t}=4\sum_{n\geq0}e^{-\frac{\pi^2(2n+1)^2}{2}t}\]

For $k=\delta$,
\begin{align*}
\int_0^{j_n} x^{1+\nu}J_{\nu}(x)dx&=[x^{1+\nu}J_{\nu+1}(x)]_0^{j_n}\\
&=j_n^{1+\nu}J_{\nu+1}(j_n)
\end{align*}

\[\phi_\delta(t)=\delta\sum_{n\geq1}e^{-\frac{j_n^2}{2}t}\]

We remark  $||\phi_\delta||_{L^2(\mathbb{R}_+)}=\infty$, which means the $\phi_k$ does not satisfy the lemma \ref{lemma} in general, and the proof of the Section 4.1. does not work to prove the following theorem.

\begin{theorem}
If $0<k<1+\delta$, then

\begin{itemize}
\item The density of the max of a $k$-generalized Bessel meander is equal to

\[\frac{\P(M\in dz)}{dz}=\frac{2}{k C_\mu z^{3+\delta-k}}\sum_{n\geq1}\left(\frac{j_{\nu,n}^{\nu+1}}{J_{\nu+1}(j_{\nu,n})}\right)^2\left(\int_0^1k x^{k-1}\frac{J_\nu(xj_n)}{(xj_n)^\nu}dx\right)e^{-\frac{j_{\nu,n}^2}{2z^2}}\]
\begin{equation}
\label{density_M_meander}+\frac{4}{k C_\mu z^{1+\delta-k}}\sum_{\substack{m,n\geq 1\\ m\neq n}}\frac{j_{\nu,m}^{\nu+1}}{J_{\nu+1}(j_{\nu,m})}\frac{j_{\nu,n}^{\nu+1}}{J_{\nu+1}(j_{\nu,n})}\left(\int_0^1k x^{k-1}\frac{J_\nu(xj_n)}{(xj_n)^\nu}dx\right)\frac{e^{-\frac{j_{\nu,n}^2}{2z^2}}-e^{-\frac{j_{\nu,m}^2}{2z^2}}}{j_{\nu,m}^2-j_{\nu,n}^2}\ (\mathcal{A})
\end{equation}
The convergence in $(\alpha_1,\alpha_2)\rightarrow(1^-,1^-)$ is uniform on every segment of $(0,\infty)$.

\item The density of the argmax of a $k$-generalized Bessel meander is equal to

\[\frac{\P(\rho\in du)}{du}=\frac{2C_{\nu-\mu}}{kC_\mu}\sum_{m,n\geq1}\frac{j_{\nu,m}^{\nu+1}}{J_{\nu+1}(j_{\nu,m})}\frac{j_{\nu,n}^{\nu+1}}{J_{\nu+1}(j_{\nu,n})}\left(\int_0^1k x^{k-1}\frac{J_\nu(xj_n)}{(xj_n)^\nu}dx\right)\]
\begin{equation}
\label{density_rho_meander}
\times\frac{1}{[j_{\nu,n}^2 u+j_{\nu,m}^2 (1-u)]^{\nu-\mu+1}}\ (\mathcal{A})
\end{equation}
The convergence in $(\alpha_1,\alpha_2)\rightarrow(1^-,1^-)$ is uniform on every segment of $(0,1)$.
\end{itemize}
\end{theorem}

\begin{proof}As for the proof of Section 4.1, we will compute those densities using (\ref{agreement_generalized_meander}) and the series expansions of $f$ and $\phi_k$. We want to show for any $0<r<R$ and $0<U<1/2$ that

\begin{align*}
\frac{1}{z^{3+\delta-k}}\int_0^1 f^{\alpha_1}\left(\frac{u}{z^2}\right)\phi_k^{\alpha_2}\left(\frac{1-u}{z^2}\right)du&\underset{\substack{\alpha_1\rightarrow1^- \\ \alpha_2\rightarrow1^-}}{\longrightarrow}\frac{1}{z^{3+\delta-k}}\int_0^1 f\left(\frac{u}{z^2}\right)\phi_k\left(\frac{1-u}{z^2}\right)du\text{ uniformly on }[r,R]\\
\int_0^\infty f^{\alpha_1}(xu) \phi_k^{\alpha_2}(x(1-u))x^{\frac{\delta-k}{2}}dx&\underset{\substack{\alpha_1\rightarrow1^- \\ \alpha_2\rightarrow1^-}}{\longrightarrow}\int_0^\infty f(xu) \phi_k(x(1-u))x^{\frac{\delta-k}{2}}dx\text{ uniformly on }[U,1-U]
\end{align*}

Since $\displaystyle{\int_0^{j_n} x^{k-1}\frac{J_\nu(x)}{x^\nu}dx=O(n^k)}$, $\phi_k$ satisfies (\ref{bounded_infinity}) and (\ref{inequality_uniform}). We study the asymptotic behavior of this integral by an integration by parts.
\[\int_1^{j_n} x^{k-1}\frac{J_\nu(x)}{x^\nu}dx=\left[-x^{k-2}\frac{J_{\nu-1}(x)}{x^{\nu-1}}\right]_1^{j_n}+(k-2)\int_1^{j_n} x^{k-3}\frac{J_{\nu-1}(x)}{x^{\nu-1}}dx\]

After such successive integrations by parts, we have a finite linear combination of $\displaystyle{\left[-x^{k-2i}\frac{J_{\nu-i}(x)}{x^{\nu-i}}\right]_1^{j_n}}$ where $i\geq1$, plus the following integral with some multiplicative constant:

\begin{align*}
\int_1^{j_n} x^{k-\nu-\lfloor k\rfloor -3}J_{\nu-\lfloor k\rfloor-2}(x)dx&=\int_1^{\infty} x^{k-\nu-\lfloor k\rfloor -3}J_{\nu-\lfloor k\rfloor-2}(x)dx-\int_{j_n}^\infty x^{k-\nu-\lfloor k\rfloor -3}J_{\nu-\lfloor k\rfloor-2}(x)dx\\
&=\int_1^{\infty} x^{k-\nu-\lfloor k\rfloor -3}J_{\nu-\lfloor k\rfloor-2}(x)dx+O\left({n}^{k-\nu-\lfloor k\rfloor -2}\right)
\end{align*}

because $k-\nu-\lfloor k\rfloor -3<-1$ and $J_{\nu-\lfloor k\rfloor-2}$ bounded on $(1,\infty)$. Moreover, for $i\geq1$,

$\displaystyle{j_n^{k-i-\nu}J_{\nu-i}(j_n)=O(n^{k-(\nu+1+1/2)})\text{ and }\frac{g_n}{j_n^{k}}=O(n^{\nu+1+1/2-k})}$. So,

\[\left(\int_0^{j_n}k x^{k-1}\frac{J_\nu(x)}{x^\nu}dx\right)\frac{g_n}{j_n^{k}}=\text{constant}\times\frac{g_n}{j_n^{k}}+O(1)\]

$\phi_k(t)$ is the sum of two terms. The first one is $\displaystyle{\sum_{n\geq1}(-1)^{n-1}\frac{g_n}{j_n^k}e^{-\frac{j_n^2}{2}t}}$ up to a multiplicative constant and the second one is dominated by $\displaystyle{\phi_\delta(t)=\delta\sum_{n\geq1}e^{-\frac{j_n^2}{2}t}}$.

Because $\displaystyle{\forall N\in\mathbb{N}^*,\ \frac{(\pi n)^k}{j_n^k}=1+\sum_{i=1}^{N-1} \text{constant(i)}\times n^{-i}+O(n^{-N})}$ and $1+(\delta-k)>0$, we can reuse the proof of Section 4.1 for the first term, and to deduce the desired convergences, we only have to control the integrals of $\displaystyle{\frac{1}{z^{3+(\delta-k)}}f\left(\frac{u}{z^2}\right)\phi_\delta\left(\frac{1-u}{z^2}\right)}$ near the boundaries. Let $\varepsilon>0$.

Note that if $q>0$, \[||\phi_\delta||_{L^2(\mathbb{R}_+,t^q dt)}\leq\delta\sqrt{2^{q+1}\Gamma(q+1)}\sum_{n\geq1}\frac{1}{j_n^{1+q}}<\infty\]

Let $0<r<R$. For any small $\eta>0$, $z\in[r,R]$, and $0<\alpha_1,\alpha_2\leq1$,
\begin{align*}
\left|\int_0^\eta f^{\alpha_1}\left(\frac{u}{z^2}\right)\phi_\delta^{\alpha_2}\left(\frac{1-u}{z^2}\right)du\right|&\leq\sqrt{\int_0^\eta f^{\alpha_1}\left(\frac{u}{z^2}\right)^2 du}\sqrt{\int_0^\eta \phi_\delta^{\alpha_2}\left(\frac{1-u}{z^2}\right)^2 du}\\
&\leq z||f^{\alpha_1}||_{L^2(\mathbb{R}_+)}\times ||\phi_\delta||_{\infty,[1/2R^2,\infty)}\sqrt{\eta}\\
&\leq RC||\phi_\delta||_{\infty,[1/2R^2,\infty)}\sqrt{\eta}\\
&\leq\varepsilon\text{ for small }\eta
\end{align*}

\begin{align*}
\left|\int_0^\eta f^{\alpha_1}\left(\frac{1-u}{z^2}\right)\phi_\delta^{\alpha_2}\left(\frac{u}{z^2}\right)du\right|&\leq\sqrt{\int_0^\eta \phi_\delta^{\alpha_2}\left(\frac{u}{z^2}\right)^2\sqrt{u} du}\sqrt{\int_0^\eta f^{\alpha_1}\left(\frac{1-u}{z^2}\right)^2\frac{1}{\sqrt{u}} du}\\
&\leq z^{3/2}||\phi_\delta||_{L^2(\mathbb{R}_+,\sqrt{t}dt)}\times ||f^{\alpha_1}||_{\infty,[1/2R^2,\infty)}\sqrt[8]{\eta}\sqrt{\int_0^1\frac{du}{u^{3/4}}}\\
&\leq 2R^{3/2}||\phi_\delta||_{L^2(\mathbb{R}_+,\sqrt{t}dt)} \sup_{0<\alpha\leq1}||f^{\alpha}||_{\infty,[1/2R^2,\infty)}\times\sqrt[8]{\eta}\\
&\leq\varepsilon\text{ for small }\eta
\end{align*}

Let $0<U<1/2$. For any small $r>0$ and big $R>0$, $u\in[U,1-U]$, and $0<\alpha_1,\alpha_2,\leq1$,
\begin{align*}
\left|\int_R^\infty f^{\alpha_1}(xu) \phi_\delta^{\alpha_2}(x(1-u))x^{\frac{\delta}{2}}dx\right|&\leq\int_R^\infty |f^{\alpha_1}(xu)| |\phi_\delta(x(1-u))|x^{\frac{\delta}{2}}dx\\
&\leq\int_R^\infty x^{\frac{\delta}{2}}e^{-\frac{j_0^2}{4}x} dx\\
&\leq\varepsilon\text{ for big }R\text{, by dominated convergence}
\end{align*}

\begin{align*}
\left|\int_0^r f^{\alpha_1}(xu) \phi_\delta^{\alpha_2}(x(1-u))x^{\frac{\delta-k}{2}}dx\right|&\leq r^{\frac{1+\delta-k}{8}}\int_0^r |f^{\alpha_1}(xu)| x^{\frac{\delta-k-1}{4}}\times|\phi_\delta^{\alpha_2}(x(1-u))|x^{\frac{1+\delta-k}{8}}  dx\\
&\leq r^{\frac{1+\delta-k}{8}}\sqrt{\int_0^r f^{\alpha_1}(xu)^2 x^{\frac{\delta-k-1}{2}} dx} \sqrt{\int_0^r \phi_\delta(x(1-u))^2 x^{\frac{1+\delta-k}{4}} dx}\\
&\leq \frac{r^{\frac{1+\delta-k}{8}}}{u^{\frac{1+\delta-k}{4}}(1-u)^{1+\frac{1+\delta-k}{8}}}C\left(\frac{1+\delta-k}{2}\right) \sqrt{\int_0^\infty \phi_\delta(x)^2 x^{\frac{1+\delta-k}{4}} dx}\\
&\leq \frac{C((1+\delta-k)/2)\times ||\phi_\delta||_{L^2(\mathbb{R}_+,t^{\frac{1+\delta-k}{4}}dt)}}{[U(1-U)]^{1+\frac{1+\delta-k}{4}}}\times r^{\frac{1+\delta-k}{8}}\\
&\leq \varepsilon\text{ for small }r\text{ because }1+\delta-k>0
\end{align*}

\end{proof}

We can simplify the expressions (\ref{density_M_meander}) and (\ref{density_rho_meander}) for some particular cases.

\begin{theorem}
If $1<\delta$, then for a classical Bessel meander, $k=2$,

\[\frac{\P(M\in dz)}{dz}=\frac{2}{ z^{\delta+1}}\sum_{n\geq1}\left(\frac{j_{\nu,n}^{\nu+1}}{J_{\nu+1}(j_{\nu,n})}\right)\left(1+\frac{1}{C_{\nu-1}}\frac{j_{\nu,n}^{\nu-1}}{J_{\nu+1}(j_{\nu,n})}\right)e^{-\frac{j_{\nu,n}^2}{2z^2}}\]
\begin{equation}
\label{density_M_2meander}+\frac{4}{ z^{\delta-1}}\sum_{\substack{m,n\geq 1\\ m\neq n}}\frac{j_{\nu,m}^{\nu+1}}{J_{\nu+1}(j_{\nu,m})}\left(1+\frac{1}{C_{\nu-1}}\frac{j_{\nu,n}^{\nu-1}}{J_{\nu+1}(j_{\nu,n})}\right)\frac{e^{-\frac{j_{\nu,n}^2}{2z^2}}-e^{-\frac{j_{\nu,m}^2}{2z^2}}}{j_{\nu,m}^2-j_{\nu,n}^2}\ (\mathcal{A})
\end{equation}

\[\frac{\P(\rho\in du)}{du}=2C_{\nu}\sum_{m,n\geq1}\frac{j_{\nu,m}^{\nu+1}}{J_{\nu+1}(j_{\nu,m})}\left(1+\frac{1}{C_{\nu-1}}\frac{j_{\nu,n}^{\nu-1}}{J_{\nu+1}(j_{\nu,n})}\right)\]
\begin{equation}
\label{density_rho_2meander}
\times\frac{1}{[j_{\nu,n}^2 u+j_{\nu,m}^2 (1-u)]^{\nu+1}}\ (\mathcal{A})
\end{equation}

\end{theorem}

\begin{theorem}
If $0<k<4$, then for the $k$-generalized Brownian meander, $\delta=3$,

\[\frac{\P(M\in dz)}{dz}=\frac{\pi^3\sqrt{2\pi}}{C_\mu z^{6-k}}\sum_{n\geq1}n^3\left(\int_0^1 x^{k-2}\sin(x\pi n)dx\right)e^{-\frac{n^2\pi^2}{2z^2}}\]
\begin{equation}
\label{density_M_bmeander}+\frac{2\pi\sqrt{2\pi}}{C_\mu z^{4-k}}\sum_{\substack{m,n\geq 1\\ m\neq n}}(-1)^{m+n}m^2n\left(\int_0^1 x^{k-2}\sin(x\pi n)dx\right)\frac{e^{-\frac{n^2\pi^2}{2z^2}}-e^{-\frac{m^2\pi^2}{2z^2}}}{m^2-n^2}\ (\mathcal{A})
\end{equation}

\[\frac{\P(\rho\in du)}{du}=\frac{\pi^{k-2}\sqrt{2\pi} C_{1/2-\mu}}{C_\mu}\sum_{m,n\geq1}(-1)^{m+n} m^2 n\left(\int_0^1 x^{k-2}\sin(x\pi n)dx\right)\]
\begin{equation}
\label{density_rho_bmeander}
\times\frac{1}{[n^2 u+m^2 (1-u)]^{3/2-\mu}}\ (\mathcal{A})
\end{equation}

\end{theorem}

\begin{theorem}
For the classical Brownian meander, $k=2$ and $\delta=3$,

\[\frac{\P(M\in dz)}{dz}=\frac{\pi^2\sqrt{2\pi}}{ z^{4}}\sum_{n\geq1}(1-(-1)^{n}) m^2 e^{-\frac{n^2\pi^2}{2z^2}}\]
\begin{equation}
\label{density_M_2bmeander}+\frac{2\sqrt{2\pi}}{ z^{2}}\sum_{\substack{m,n\geq 1\\ m\neq n}}((-1)^{m+n}-(-1)^{m})\times m^2\frac{e^{-\frac{n^2\pi^2}{2z^2}}-e^{-\frac{m^2\pi^2}{2z^2}}}{m^2-n^2}\ (\mathcal{A})
\end{equation}

\begin{equation}
\label{density_rho_2bmeander}
\frac{\P(\rho\in du)}{du}=2\sum_{m,n\geq1}(-1)^{m-1}\frac{m^2}{[(2n-1)^2 u+m^2 (1-u)]^{3/2}}\ (\mathcal{A})
\end{equation}

\end{theorem}

\begin{theorem}
For a Bessel process on $[0,1]$, $0<k=\delta$,

\begin{equation}
\label{density_M_pbessel}
\frac{\P(M\in dz)}{dz}=\frac{2}{C_\nu z^{3}}\sum_{n\geq1}\left(\frac{j_{\nu,n}^{\nu+1}}{J_{\nu+1}(j_{\nu,n})}\right) e^{-\frac{j_{\nu,n}^2}{2z^2}}+\frac{4}{C_\nu z}\sum_{\substack{m,n\geq 1\\ m\neq n}}\frac{j_{\nu,m}^{\nu+1}}{J_{\nu+1}(j_{\nu,m})}\frac{e^{-\frac{j_{\nu,n}^2}{2z^2}}-e^{-\frac{j_{\nu,m}^2}{2z^2}}}{j_{\nu,m}^2-j_{\nu,n}^2}\ (\mathcal{A})
\end{equation}

\begin{equation}
\label{density_rho_pbessel}
\frac{\P(\rho\in du)}{du}=\frac{2}{C_\nu}\sum_{m,n\geq1}\frac{j_{\nu,m}^{\nu+1}}{J_{\nu+1}(j_{\nu,m})}\frac{1}{j_{\nu,n}^2 u+j_{\nu,m}^2 (1-u)}\ (\mathcal{A})
\end{equation}

\end{theorem}

\begin{remark}
The formulas (\ref{density_M_2bmeander}) and (\ref{density_rho_2bmeander}) have been found in \cite{majumdar} after some non-rigorous manipulations (see our introduction). The formula (\ref{density_M_pbessel}) has been discussed in \cite{MR1701890}. But for the case of the maximum of a Bessel process, we can compute the density much more simply. Indeed, $\displaystyle{\P(M\leq z)=\P(T_z\geq 1)=\P\left(T\geq\frac{1}{z^2}\right)}$, so $\displaystyle{\frac{\P(M\in dz)}{dz}=\frac{2}{z^3}f\left(\frac{1}{z^2}\right)}$.
\end{remark}

\section{Conclusion and perspectives}

Thanks to the decomposition at the maximum, it is possible to give the densities of $M$ and $\rho$ for Bessel bridges and skew Brownian bridges as double Abel summation. While the law of $M$ is have been often studied (in \cite{MR1701890} for Bessel bridges for example), much less is known on the law of $\rho$. More details such as its moments or its Mellin transform could be examined thanks to the provided density.

For the bridges, the decomposition at the maximum is very satisfying because it makes appear two independent parts with almost the same law of the path. In particular, the lengths of the two parts are i.i.d and are almost explicit in the decomposition, their law is just a first hitting time of $1$, up to a scaling of the total length.

The case of the classical Brownian meander is quite appealing because it shows how things are worst when the end point is not fixed. The two parts of the decomposition are again independent but obviously not of the same law. More disturbing, the second part can be viewed as if its starting point is chosen before running it. So the length of this part is not explicit and its density is an integral of densities corresponding to the possible starting points.

However, in the case of the meander, the law of $S$ is surprisingly not too exotic. In fact, $4S$ has the same law as the last zero of a Brownian motion before hitting $+-1$. Is it only a pretty coincidence, or is it coming from something deeper ? Reformulating the question, could we find an other decomposition for a class of processes making the end point or the time between the end of the path and the time of the maximum explicit ?

A further interest could be finding different ways to deduce the joint density of $(M,\rho)$ from decomposition at the maximum. For example, the articles \cite{majumdar} and \cite{000274266600015n.d.} present how the path-integral method and the real space renormalization group lead to that density, even if they lack some technical arguments. A stronger approach may be used to unify the methods used for Bessel bridges and for skew Brownian bridges in our work, and even to generalize those results to other processes.

\section*{Acknowledgments}

This research was supported by a grant from the Ecole Normale Sup{\'e}rieure, Paris which gave me the opportunity to visit the Department of Statistics at the University of California, Berkeley from April to July 2018. The research was done under the supervision of Jim Pitman. I am grateful to him for proposing this work, for always being patient and helpful, and for introducing me to different aspects of mathematics. I also thank Satya Majumdar and Michael Kearney for their advice and for drawing my attention to \cite{000274266600015n.d.}.

\bibliographystyle{plain}
\bibliography{biblio}
\end{document}